\documentclass{amsproc}

\usepackage{amsmath, amssymb, amsthm, amsfonts, enumerate, color, comment}

\usepackage{palatino}
\usepackage{graphicx}

\usepackage{parskip}

\numberwithin{equation}{section}
\theoremstyle{plain}
\newtheorem{Proposition}[equation]{Proposition}
\newtheorem{Corollary}[equation]{Corollary}
\newtheorem*{Corollary*}{Corollary}
\newtheorem{Theorem}[equation]{Theorem}
\newtheorem*{Theorem*}{Theorem}

\theoremstyle{definition}
\newtheorem{Definition}[equation]{Definition}

\usepackage{enumitem}
\setlist[enumerate]{leftmargin=*}
\setlist[itemize]{leftmargin=*}

\def\C{\mathbb{C}}
\def\R{\mathbb{R}}
\def\D{\mathbb{D}}
\def\T{\mathbb{T}}
\def\N{\mathbb{N}}
\def\Z{\mathbb{Z}}

\renewcommand{\leq}{\leqslant}
\renewcommand{\geq}{\geqslant}
\renewcommand{\subset}{\subseteq}
\renewcommand{\phi}{\varphi}
\renewcommand{\vec}[1]{{\bf #1}}

\usepackage{xcolor}
\newcommand{\0}{{\color{lightgray}0}}

\author[W. Ross]{William T. Ross}
	\address{Department of Mathematics and Statistics, University of Richmond, Richmond, VA 23173, USA}
	\email{wross@richmond.edu}
	

\title{The Ces\`{a}ro Operator}


\thanks{}

\begin{document}

\begin{abstract}
This paper surveys the various aspacts of the Ces\`{a}ro operator with a special emphasis on the Hilbert space setting of $\ell^2$. We include a discussion of summability methods of Fourier series, the Ces\`{a}ro matrix and integral operator, and the Ces\`{a}ro operator on the Hilbert space $\ell^2$. In this setting of $\ell^2$ we cover the spectral properties of the Ces\`{a}ro operator as well as a treatment of its hyponormality (via posinormal operators) and subnormality (via a multiplication operator of Kriete and Trutt). Using the results of Kriete and Trutt, we describe the commutant of the Ces\`{a}ro operator as well as its bounded square roots. The Ces\`{a}ro operator has a rich lattice of invariant subspaces whose description remains unknown and we suspect might never be fully understood. We survey some results which display the complexity of these invariant subspaces and provide the reader with several paths forward for further discussion. Though the  Ces\`{a}ro operator was originally explored in the $\ell^2$ and Hardy space settings, it has been explored in various other settings such as the general Hardy spaces as well as the Bergman spaces. Finally, we explore some of the so-called generalized Ces\`{a}ro operator which are classes of integral operators on the Hardy spaces that connect to many areas of classical function theory. 
\end{abstract}

\maketitle

\section{Introduction}

There has been renewed interest in the classical Ces\`{a}ro operator and its generalizations as of late \cite{AglMc, PartGall, MR4365517, MR3342485, MR2390921, MR4366190} so perhaps it is a good time to put together an extended survey of what is currently known about this operator. Most of us in analysis know the name Ces\`{a}ro from his summability method for infinite series and the important role this plays in summing the Fourier series of an integrable function. We will outline some of the aspects of Ces\`{a}ro's life and the historical development of his summability method in \S \ref{CesaroSum}. 

Graduating from Ces\`{a}ro summability, which involves creating the sequence of Ces\`{a}ro averages
$$\Big(\frac{a_0 + a_1 + a_2 + \cdots + a_{N}}{N + 1}\Big)_{N \geq 0}$$ from the sequence $(a_n)_{n \geq 0}$ of complex numbers, to considering the linear transformation 
$$(a_n)_{n \geq 0} \mapsto \Big(\frac{a_0 + a_1 + a_2 + \cdots + a_{N}}{N + 1}\Big)_{N \geq 0}$$
on the sequence space $\ell^2$, was taken up by Brown, Halmos, and Shields \cite{MR187085}. By representing this linear transformation with respect to the standard orthonormal basis $(\vec{e}_n)_{n \geq 0}$ for $\ell^2$, as the Ces\`{a}ro matrix 
$$C := \begin{bmatrix}
1 & \0 & \0 & \0 & \0 &  \cdots\\[3pt]
\frac{1}{2} & \frac{1}{2} & \0 & \0 & \0 & \cdots\\[3pt]
\frac{1}{3} & \frac{1}{3} & \frac{1}{3} & \0 & \0 & \cdots\\[3pt]
\frac{1}{4} & \frac{1}{4} & \frac{1}{4} & \frac{1}{4} & \0 & \cdots\\[3pt]
\frac{1}{5} & \frac{1}{5} & \frac{1}{5} & \frac{1}{5} & \frac{1}{5} & \cdots\\
\vdots & \vdots & \vdots & \vdots & \vdots & \ddots
\end{bmatrix},$$
they were able to describe its (operator) norm $\|C\|$ and its spectrum $\sigma(C)$. We survey this in \S\ref{cvgsdhfjgksa111}
and \S \ref{spectrusms8}. Along the way, we discuss the Ces\`{a}ro matrix as a linear transformation on the set of {\em all} complex sequences in \S \ref{section444}.

Brown, Halmos, and Shields also showed that the Ces\`{a}ro operator on $\ell^2$ is hyponormal, i.e.,  $C^{*}C - C C^{*} \geq 0$ (with, of course, normal being equality). By representing the Ces\`{a}ro operator as a multiplication operator on the closure of the polynomials in an $L^2(\mu)$ space, Kriete and Trutt \cite{MR281025}, in one of the gems in the subject, showed that the Ces\`{a}ro operator satisfies the stronger condition of subnormality (a normal operator restricted to one of its invariant subspaces). This will be outlined in \S \ref{HN} and \S \ref{KT}. 

As a consequence of the Kriete--Trutt result, and some known results concerning the commutant of a multiplication operator on a reproducing kernel Hilbert space of analytic functions, one can describe the bounded operators on $\ell^2$ that commute with the Ces\`{a}ro operator. By a result of Shields and Wallen \cite{MR287352}, these operators coincide with the closure in the weak operator topology of $p(C)$, where $p \in \C[z]$. This is discussed in \S \ref{commuette} -- with a corresponding discussion  in \S \ref{section444}  of the linear transformations on the space of all complex sequences that commute with the Ces\`{a}ro matrix.  We use this discussion in \S \ref{sqrt} to show that the Ces\`{a}ro operator has exactly two bounded square roots and many other unbounded ones. 

The Ces\`{a}ro operator has several interesting and useful connections to strongly continuous semi-groups of composition operators on the Hardy space $H^2$. We explore some of these connections in \S \ref{CCO} and then again in \S \ref{ISCOp} when we discuss the complexity of the invariant subspaces of the Ces\`{a}ro operator. 

We end this survey with a selective survey in \S \ref{Bergman} and \S \ref{g} of the Ces\`{a}ro operator on other spaces of analytic functions along with generalizations of the Ces\`{a}ro operator that are inspired by its integral representation on the Hardy space $H^2$.

\section{Ces\`{a}ro summability}\label{CesaroSum}

Many people who work on the Ces\`{a}ro operator are not familiar with its namesake Ernesto Ces\`{a}ro who was born in Naples, Italy on March 12, 1859 and died tragically, while trying to save his son from drowning, in Torre Annunziata, Italy, on September 12, 1906. He began his university studies in \`{E}cole des Mines of Li\`{e}ge (in Belgium) under Eug\`{e}ne Catalan and continued his studies in Rome. He held university positions in Palermo, Naples, and Bologna. 

Though readers of this paper probably know Ces\`{a}ro for his concept of summability of infinite series, one might not know that Ces\`{a}ro was also a prolific author with varied interests in several hundred papers in analysis, geometry, number theory, and mathematical physics. 

Traditionally, one sums an infinite series $a_0 + a_1 + a_2 + \cdots$ of complex numbers by considering
\begin{equation}\label{classiffgsummera}
\lim_{N \to \infty} S_N, \quad S_{N} := \sum_{j = 0}^{N} a_j,
\end{equation}
the limit of the partial sums $S_N$. When the limit exists, we say the series converges and its value is denoted by 
$\sum_{n = 0}^{\infty} a_n$. When the limit $\lim_{N \to \infty} S_N$ fails to exist, we declare the series $a_0 + a_1 + a_2 + \cdots$ to be divergent. In some sense, this definition of convergence is rather restrictive (and perhaps capricious) in that one runs into trouble with the very natural problem of summing a Fourier series. It is well known that for an $f \in L^1 := L^1(\T, m)$ there are technical problems in summing the Fourier series 
$$\sum_{n = -\infty}^{\infty} \widehat{f}(n) \xi^{n}.$$ In the above, $\T = \{\xi \in \C: |\xi| = 1\}$ is the unit circle in the complex plane $\C$, $m$ is standard normalized Lebesgue measure (i.e., $d \theta/2\pi$) on $\T$, and 
$$\widehat{f}(n) := \int_{\T} f(\xi) \overline{\xi}^{n} dm(\xi), \quad n \in \Z,$$
is the $n$th Fourier coefficient of $f$. An example of du Bois-Reymond from 1873 \cite{Reymond1873} produces a continuous function $\T$ whose Fourier series diverges at $\xi = 1$. An even more pathological example of Kolmogorov from 1923 \cite{KolFS} produces an $L^1$ function whose Fourier series diverges at {\em every} point of $\T$. In spite of these pathologies, one still needs to somehow relate a function with its Fourier series in a meaningful way. One way out of this conundrum (there are several) involves a summation method of Ces\`{a}ro \cite{Cesaro1890}.

\begin{Definition}
For a sequence $(a_n)_{n \geq 0}$ of complex numbers, we say the infinite series $\sum_{n = 0}^{\infty} a_n$ is {\em Ces\`{a}ro summable} to $L$ if 
$$L = \lim_{N \to \infty} \frac{1}{N + 1} \sum_{n = 0}^{N} S_n.$$
\end{Definition}

Notice how $L$ is the limit of the averages of the partial sums. The reader might wonder why we are indexing our sequences starting with $a_0$ instead of $a_1$. As we shall see momentarily, when we recast the Ces\`{a}ro operator from the setting of vector spaces of sequences  to that of vector spaces of analytic functions, we will identify the sequence $(a_n)_{n \geq 0}$ with the power series $a_0 + a_1 z + a_2 z^2 + \cdots$.

The reader can verify that the famous {\em Grandi} series 
$$1 - 1 + 1 - 1 + 1 - \cdots$$ diverges in the ``usual sense'' of \eqref{classiffgsummera} but is Ces\`{a}ro summable to $\tfrac{1}{2}$. A nice real analysis exercise will confirm the following. 

\begin{Proposition}
If $(a_n)_{n \geq 0}$ is a sequence of complex numbers for which $S_{N} \to L$, then the series $\sum_{n = 0}^{\infty} a_n$ is Ces\`{a}ro summable to $L$.
\end{Proposition}


Well known theorems of Fej\'{e}r and Lebesgue \cite{Hoffman} show that although the Fourier series of an $L^1$ function need not converge in the usual sense, it does converge in the sense of Ces\`{a}ro. 

\begin{Theorem}\label{Fejer}
The Fourier series of  an $f \in L^1$ is Ces\`{a}ro summable to $f$ almost everywhere on $\T$.
\end{Theorem}

The proof of this theorem is often the starting point of just about any harmonic analysis course. Indeed, if 
$$(S_{n} f)(\xi) := \sum_{k = -n}^{n} \widehat{f}(n) \xi^k$$
is
the $n$th partial sum of the Fourier series for $f$, then 
$$(S_{n} f)(\xi) = \int_{\T} f(w) \Big(\sum_{k = -n}^{n} (\xi \overline{w})^k\Big) dm(w).$$
Thus,
$$(\sigma_N f)(\xi):= \frac{1}{N + 1} \sum_{n = 0}^{N} (S_{n} f)(\xi),$$
the $N$th Ces\`{a}ro mean of the Fourier series of $f$, becomes 
$$\int_{\T} f(w) \Big(\frac{1}{N + 1} \sum_{n = 0}^{N} \sum_{k = -n}^{n} (\xi \overline{w})^k\Big) dm(w).$$
The quantity 
$$ \frac{1}{N + 1} \sum_{n = 0}^{N} \sum_{k = -n}^{n} (\xi \overline{w})^k$$
in the integral above
is known as the Fej\'{e}r kernel and has the requisite properties needed to prove Theorem \ref{Fejer}.

Before moving on, let us mention that properties of the Fej\'{e}r kernel show that if $f$ is continuous on $\T$, then $\sigma_{N}(f) \to f$ uniformly (which gives a proof of the Weierstrass theorem on the density of the trigonometric polynomials in the Banach space of continuous functions on $\T$). Furthermore,  if $f \in L^p$, $1 \leq p < \infty$, then $\sigma_{N}(f) \to f$ in $L^p$ norm. If $f \in L^{\infty}$, then $\sigma_{N}(f) \to f$ weak-$*$.

There are higher order summability methods for a series $a_0 + a_1 + a_2 + \cdots$.  Define 
$$H_{N}^{0} := \sum_{j = 0}^{N} a_j,$$
$$H_{N}^{1} := \frac{H_{0}^{0} + H_{1}^{0} + H_{2}^{0} + \cdots + H_{N}^{0}}{N + 1},$$
$$H_{N}^{2} := \frac{H_{0}^{1} + H_{1}^{1} + H_{2}^{1} + \cdots + H_{N}^{1}}{N + 1},$$ and for $r \in \N$, 
$$H_{N}^{r} := \frac{H_{0}^{r - 1} + H_{1}^{r - 1} + H_{2}^{r - 1} + \cdots + H_{N}^{r - 1}}{N + 1}.$$ 
We say a series $a_0 + a_1 + a_2 + \cdots$ is {\em H\"{o}lder summable of order $r$}, written $(H, r)$ summable, to $L$ if
$$\lim_{N \to \infty} H_{N}^{r} = L.$$
See \cite{Holder} for the original reference. 

Let 
$$S_{N}^{0} := a_0 + a_1 + a_2 + \cdots + a_{N},$$
$$S_{N}^{1} := S_{0}^{0} + S_{1}^{0} + S_{2}^{0} + \cdots + S_{N}^{0},$$
$$S_{N}^{2} := S_{0}^{1} + S_{1}^{1} + S_{2}^{1} + \cdots + S_{N}^{1},$$
and for $r \in \N$, 
$$S_{N}^{r} := S_{0}^{r - 1} + S_{1}^{r - 1} + S_{2}^{r - 1} + \cdots + S_{N}^{r - 1}.$$
We say a series $a_0 + a_1 + a_2 + \cdots$ is {\em Ces\`{a}ro summable of order $r$}, written $(C, r)$ summable, to $L$ if 
$$\lim_{N \to \infty} \frac{r! S_{N}^{r}}{(N + 1)^r} = L.$$
Note that $(H, 1)$ and $(C, 1)$ summability are the same. It is also known that if a series is $(H, s)$  summable to $L$ for some $s$, then it is $(H, r)$ summable to $L$ for all $r \geq s$ (similarly for $(C, r)$ summability). In fact, a series is $(H, r)$ summable to $L$ if and only if it is $(C, r)$ is summable to $L$. We will revisit these summability methods as matrices momentarily. 

For example, the $(C, p + 1)$ sum of the series 
$$1^p - 2^p + 3^p - 4^p + \cdots$$
is equal to 
$$\frac{2^{p + 1} - 1}{p + 1} B_{p + 1},$$
where $B_{p + 1}$ is the corresponding Bernoulli number.
Note that when $p = 0$, we have $B_{1} = \tfrac{1}{2}$ and $(C, 1)$ sum of the series 
$1 - 1 + 1 - 1 + 1 + \cdots$ is $\tfrac{1}{2}$ which we discussed earlier. When $p = 1$, we have $B_{2} = \tfrac{1}{6}$ and the $(C, 2)$ sum of the series $1 - 2 + 3 - 4 + 5 - \cdots$ is $\tfrac{1}{4}$.
 In fact, Ces\`{a}ro found these formulas curious and said \cite{Camazing}: ``I say that, although these formulas are false, they may be used as a base of a theory, which shall not be more absurd than the theory of imaginaries.''

For further reading on all things concerning both summability methods of series of constants as well as Fourier series, we direct the reader to two well-known texts. The first is Hardy's classic text \cite{Hardy6} on various aspects of divergent series and summability methods. The second is Zygmund's thorough treatment of Fourier series \cite{MR1963498}. For a well-documented history behind what Ces\`{a}ro was thinking when he developed his summability method, the paper \cite{MR1705301} is an excellent resource. 

\section{The Ces\`{a}ro matrix}\label{section444}

For a sequence $(a_n)_{n \geq 0}$ of complex numbers,  the Ces\`{a}ro averages
$$\sigma_{N} := \frac{1}{N + 1} \big(S_0 + S_1 + S_2 + \cdots + S_{N}\big)$$ 
can be obtained via a linear transformation on the vector space $\mathcal{V}$ of all complex sequences as follows. By means of the {\em Ces\`{a}ro matrix}
$$C := \begin{bmatrix}
1 & \0 & \0 & \0 & \0 &  \cdots\\[3pt]
\frac{1}{2} & \frac{1}{2} & \0 & \0 & \0 & \cdots\\[3pt]
\frac{1}{3} & \frac{1}{3} & \frac{1}{3} & \0 & \0 & \cdots\\[3pt]
\frac{1}{4} & \frac{1}{4} & \frac{1}{4} & \frac{1}{4} & \0 & \cdots\\[3pt]
\frac{1}{5} & \frac{1}{5} & \frac{1}{5} & \frac{1}{5} & \frac{1}{5} & \cdots\\[-3pt]
\vdots & \vdots & \vdots & \vdots & \vdots & \ddots
\end{bmatrix},$$
observe that 
\begin{equation}\label{zppsooIDIIFUUG}
\begin{bmatrix}
1 & \0 & \0 & \0 & \0 &  \cdots\\[3pt]
\frac{1}{2} & \frac{1}{2} & \0 & \0 & \0 & \cdots\\[3pt]
\frac{1}{3} & \frac{1}{3} & \frac{1}{3} & \0 & \0 & \cdots\\[3pt]
\frac{1}{4} & \frac{1}{4} & \frac{1}{4} & \frac{1}{4} & \0 & \cdots\\[3pt]
\frac{1}{5} & \frac{1}{5} & \frac{1}{5} & \frac{1}{5} & \frac{1}{5} & \cdots\\[-3pt]
\vdots & \vdots & \vdots & \vdots & \vdots & \ddots
\end{bmatrix}
 \begin{bmatrix} 
S_0\\[3pt]
S_1\\[3pt]
S_2\\[3pt]
S_3\\[3pt]
S_4\\[3pt]
\vdots\end{bmatrix} 
=  \begin{bmatrix} 
S_0\\[3pt]
\tfrac{1}{2} (S_0 + S_1)\\[3pt]
\tfrac{1}{3} (S_0 + S_1 + S_2)\\[3pt]
\tfrac{1}{4} (S_0 + S_1 + S_2 + S_3)\\[3pt]
\tfrac{1}{5} (S_0 + S_1 + S_2 + S_3 + S_4)\\[3pt]
\vdots\end{bmatrix}= 
\begin{bmatrix}
\sigma_0\\[3pt]
\sigma_1\\[3pt]
\sigma_2\\[3pt]
\sigma_3\\[3pt]
\sigma_4\\[3pt]
\vdots
\end{bmatrix}.
\end{equation}

Not only does the matrix $C$ produce the Ces\`{a}ro partial sums, it is the starting point in studying the linear transformation $\vec{a} \mapsto C \vec{a}$ on $\mathcal{V}$. Let us first explore the eigenvalues and eigenvectors of $C$ on $\mathcal{V}$. 
 For each $m \in \N_{0} := \N \cup \{0\}$, define the sequence (as a column vector) 
\begin{equation}\label{bmmmmm}
\vec{b}_{m} = \begin{bmatrix} 0 & 0 & 0 & \cdots & 0 & {m + 0 \choose m} & {m  + 1\choose m} & {m  + 2\choose m} & {m  + 3\choose m} & \cdots\end{bmatrix}^{T}.
\end{equation}

\begin{Proposition}\label{EVEV}
For each $m \in \N_0$, 
$$C \vec{b}_m  = \frac{1}{m + 1} \vec{b}_m.$$
Moreover, each eigenspace is one dimensional. 
\end{Proposition}

We will outline a proof of this in the next section when we re-imagine the Ces\`{a}ro operator as an integral operator on the space of all analytic functions on $\D$. Next observe that the Ces\`{a}ro matrix takes $\mathcal{V}$, the vector space of all complex sequences, {\em onto} $\mathcal{V}$. 

\begin{Proposition}
The Ces\`{a}ro matrix defines an invertible linear transformation of $\mathcal{V}$ onto $\mathcal{V}$ with
$$C^{-1} = 
\begin{bmatrix}
\phantom{-}1 & \phantom{-} \0 & \phantom{-} \0 & \phantom{-} \0 & \phantom{-} \0 & \cdots\\
-1 & \phantom{-} 2 & \phantom{-} \0 & \phantom{-} \0 &\phantom{-}  \0 &  \cdots\\
\phantom{-} \0 & -2 & \phantom{-} 3 & \phantom{-} \0 & \phantom{-} \0 & \cdots\\
\phantom{-} \0 & \phantom{-} \0 & -3 & \phantom{-} 4 & \phantom{-} \0 & \cdots\\
\phantom{-} \0 & \phantom{-} \0 & \phantom{-} \0 & -4 & \phantom{-} 5 & \cdots\\
\phantom{-} \vdots & \phantom{-} \vdots & \phantom{-} \vdots & \phantom{-} \vdots & \phantom{-} \vdots & \ddots
\end{bmatrix}.
$$
\end{Proposition}

One proves this by referring to \eqref{zppsooIDIIFUUG} and solving the infinite triangular system 
\begin{align*}
\sigma_0 & = S_0\\
\sigma_1 & = \tfrac{1}{2} (S_0 + S_1)\\
\sigma_{2} & = \tfrac{1}{3} (S_0 + S_1 + S_2)\\
\sigma_{3} & = \tfrac{1}{4} (S_0 + S_1 + S_2 + S_3)
\end{align*}
and so on, for $S_0, S_1, S_2, \ldots$

Interestingly, the Ces\`{a}ro matrix connects with some other famous matrices. These factorizations where pointed out by Bennett \cite{MR858964, MR1317938} as a way or estimating, and even computing, the norms of various well-known matrices and their generalizations. For example, if $H$ is the classical {\em Hilbert matrix}
$$H = \begin{bmatrix}
1 & \tfrac{1}{2} & \frac{1}{3} & \frac{1}{4} & \cdots  \\[4pt]
\frac{1}{2} & \frac{1}{3} & \frac{1}{4} & \frac{1}{5} & \cdots  \\[4pt]
\frac{1}{3} & \frac{1}{4} & \frac{1}{5} & \frac{1}{6} & \cdots  \\
\vdots & \vdots & \vdots & \vdots & \ddots
\end{bmatrix}, $$
i.e., $H = [H_{jk}]_{j,k \geq 0}$, where 
$$H_{jk} = \frac{1}{j + k + 1},$$ and 
$$B =
\begin{bmatrix}
 \frac{1}{2} & \frac{1}{3} & \frac{1}{4} & \frac{1}{5} & \frac{1}{6} & \frac{1}{7} & \cdots  \\[4pt]
 \frac{1}{6} & \frac{1}{6} & \frac{3}{20} & \frac{2}{15} & \frac{5}{42} & \frac{3}{28} & \cdots  \\[4pt]
 \frac{1}{12} & \frac{1}{10} & \frac{1}{10} & \frac{2}{21} & \frac{5}{56} & \frac{1}{12} & \cdots  \\[4pt]
 \frac{1}{20} & \frac{1}{15} & \frac{1}{14} & \frac{1}{14} & \frac{5}{72} & \frac{1}{15} & \cdots  \\[4pt]
 \frac{1}{30} & \frac{1}{21} & \frac{3}{56} & \frac{1}{18} & \frac{1}{18} & \frac{3}{55} & \cdots  \\[4pt]
 \frac{1}{42} & \frac{1}{28} & \frac{1}{24} & \frac{2}{45} & \frac{1}{22} & \frac{1}{22} & \cdots  \\[4pt]
 \vdots & \vdots & \vdots & \vdots & \vdots & \vdots & \ddots
 \end{bmatrix},$$
 i.e., $B = [B_{jk}]_{j, k \geq 0},$ where
 $$B_{j k} = \frac{k + 1}{(j + k + 1)(j + k + 2)},$$
 one can show \cite{MR858964} that 
 \begin{equation}\label{HCe}
 H = B C
 \end{equation}
 which links the Hilbert matrix with the Ces\`{a}ro matrix.

 If one considers the transpose
$$\begin{bmatrix}
1 & \frac{1}{2} & \frac{1}{3} & \frac{1}{4} & \frac{1}{5} &  \cdots\\[3pt]
\0 & \frac{1}{2} & \frac{1}{3} & \frac{1}{4} & \frac{1}{5} & \cdots\\[3pt]
\0 & \0 & \frac{1}{3} & \frac{1}{4} & \frac{1}{5} & \cdots\\[3pt]
\0 & \0 & \0 & \frac{1}{4} & \frac{1}{5} & \cdots\\[3pt]
\0 & \0 & \0 & \0 & \frac{1}{5} & \cdots\\[-3pt]
\vdots & \vdots & \vdots & \vdots & \vdots & \ddots
\end{bmatrix}$$
of $C$ and defines
$$L := 
\begin{bmatrix}
1 & \tfrac{1}{2} & \tfrac{1}{3} & \tfrac{1}{4} & \cdots\\[5pt]
\tfrac{1}{2} & \tfrac{1}{2} & \tfrac{1}{3} & \tfrac{1}{4} & \cdots\\[5pt]
\tfrac{1}{3} & \tfrac{1}{3} & \tfrac{1}{3} & \tfrac{1}{4}&  \cdots\\[5pt]
\tfrac{1}{4} & \tfrac{1}{4} & \tfrac{1}{4} & \tfrac{1}{4} & \cdots\\
\vdots & \vdots & \vdots & \vdots & \ddots
\end{bmatrix},$$
in other words, $L = [L_{i j}]_{i, j \geq 0}$, with 
$$L_{i j} := \frac{1}{\max\{i, j\} + 1},$$
a computation will show that $L = C C^{*}$. The matrix above is a special case of an {\em $L$-matrix} and was explored recently in \cite{BLM2, BLM1}.

As with any matrix, one can ask about the infinite matrices $T$ for which $T C =  C T$, i.e., the commutant of $C$ as a linear transformation on $\mathcal{V}$. These involve the {\em Hausdorff matrices} which are constructed as follows. Define the matrix 
$$W := 
\begin{bmatrix}
\phantom{-}1 & \phantom{-}\0 &\phantom{-} \0 & \phantom{-}\0 & \phantom{-}\0 & \phantom{-}\0 & \cdots\\
\phantom{-}1 & -1 & \phantom{-}\0 & \phantom{-}\0 &\phantom{-} \0 & \phantom{-}\0 & \cdots\\
\phantom{-}1 & -2  & \phantom{-}1 &\phantom{-} \0 &\phantom{-} \0 & \phantom{-}\0 & \cdots\\
\phantom{-}1 & - 3  &  \phantom{-}3  & -1 &\phantom{-} \0 & \phantom{-}\0 & \cdots\\
\phantom{-}1 & -4  &\phantom{-} 6  & -4  & \phantom{-}1 & \phantom{-}\0 & \cdots\\
\phantom{-}1 & -5  & \phantom{-}10  & -10  & \phantom{-}5  & -1& \cdots\\
\phantom{-}\vdots & \phantom{-}\vdots &\phantom{-} \vdots & \phantom{-}\vdots & \phantom{-}\vdots & \phantom{-}\vdots & \ddots
\end{bmatrix}.$$
Notice how each row consists of the alternating binomial coefficients. A matrix calculation (which is allowed since $W$ is lower triangular) shows that 
$$W^2 = I.$$
Furthermore, and this will be used in \S \ref{commuette}, one can verify the matrix identity \cite{MR1544453}
\begin{equation}\label{Haus}
C = W D W,
\end{equation}
 where $D$ is the diagonal matrix 
$$D := \begin{bmatrix}
1 & \0 & \0 & \0 & \cdots\\[3pt]
\0 & \frac{1}{2} & \0 & \0 & \cdots\\[3pt]
\0 & \0 & \frac{1}{3} & \0 & \cdots\\[3pt]
\0 & \0 & \0 & \frac{1}{4} & \cdots\\
\vdots & \vdots & \vdots & \vdots & \ddots
\end{bmatrix}.$$
If $T C = C T$, then the matrix $T' := W TW$ commutes with $D$ and, since $D$ is a diagonal matrix with distinct entries, one can argue that 
$$T' = \begin{bmatrix}
z_0 & \0 & \0 & \0 & \cdots\\[3pt]
\0 & z_1 & \0 & \0 & \cdots\\[3pt]
\0 & \0 & z_2 & \0 & \cdots\\[3pt]
\0 & \0 & \0 & z_3 & \cdots\\
\vdots & \vdots & \vdots & \vdots & \ddots
\end{bmatrix}$$ for some sequence $(z_n)_{n \geq 0}$ of complex numbers. Thus, 
$$T = W  \begin{bmatrix}
z_0 & \0 & \0 & \0 & \cdots\\[3pt]
\0 & z_1 & \0 & \0 & \cdots\\[3pt]
\0 & \0 & z_2 & \0 & \cdots\\[3pt]
\0 & \0 & \0 & z_3 & \cdots\\
\vdots & \vdots & \vdots & \vdots & \ddots
\end{bmatrix} W.$$ 

This yields the following. 

\begin{Theorem}[Hurwitz--Silverman \cite{MR1501058}]\label{sdlfuuHHJHJHH}
A matrix $T$ commutes with the Ces\`{a}ro matrix $C$ if and only if 
$$T = W  \begin{bmatrix}
z_0 & \0 & \0 & \0 & \cdots\\[3pt]
\0 & z_1 & \0 & \0 & \cdots\\[3pt]
\0 & \0 & z_2 & \0 & \cdots\\[3pt]
\0 & \0 & \0 & z_3 & \cdots\\
\vdots & \vdots & \vdots & \vdots & \ddots
\end{bmatrix} W$$ for some complex sequence  $(z_n)_{n \geq 0}$.
\end{Theorem}

The matrix $T$ above is known as a {\em Hausdorff matrix}. 
An interesting class of Hausdorff matrices are the {\em Euler matrices} $E_{\lambda}$, where $\lambda \in \C$ and 
\begin{equation}\label{Euler}
E_{\lambda} :=  W  \begin{bmatrix}
1 & \0 & \0 & \0 & \cdots\\[3pt]
\0 & \lambda & \0 & \0 & \cdots\\[3pt]
\0 & \0 & \lambda^2& \0 & \cdots\\[3pt]
\0 & \0 & \0 & \lambda^3 & \cdots\\
\vdots & \vdots & \vdots & \vdots & \ddots
\end{bmatrix} W.
\end{equation}
We will see these again in \S \ref{commuette}. The above Hausdorff operators have various generalizations that were explored in \cite{MR305108}. 

Before leaving this section, let us connect these Hausdorff matrices with higher order H\"{o}lder and Ces\`{a}ro summability discussed earlier. For an infinite  series  $a_0 + a_1 + a_2 + \cdots$, recall the quantities 
$$H_{N}^{0} := \sum_{j = 0}^{N} a_j,$$
$$H_{N}^{r} := \frac{H_{0}^{r - 1} + H_{1}^{r - 1} + H_{2}^{r - 1} + \cdots + H_{N}^{r - 1}}{N + 1}, \quad r \in \N.$$ 
Also recall that a series $a_0 + a_1 + a_2 + \cdots$ is H\"{o}lder summable or order $r$, written $(H, r)$ summable, to $L$ if
$$\lim_{N \to \infty} H_{N}^{r} = L.$$
The matrix that corresponds to $(H, r)$ summability, as the Ces\`{a}ro matrix corresponded to Ces\`{a}ro summability in \eqref{zppsooIDIIFUUG}, is 
$$W  \begin{bmatrix}
f(1) & \0 & \0 & \0 & \cdots\\[3pt]
\0 & f(\tfrac{1}{2}) & \0 & \0 & \cdots\\[3pt]
\0 & \0 & f(\tfrac{1}{3}) & \0 & \cdots\\[3pt]
\0 & \0 & \0 & f(\tfrac{1}{4}) & \cdots\\
\vdots & \vdots & \vdots & \vdots & \ddots
\end{bmatrix} W,$$
where $f(z) = z^r$.

Similarly, recall the quantities 
$$S_{N}^{0} := a_0 + a_1 + a_2 + \cdots + a_{N},$$
$$S_{N}^{r} := S_{0}^{r - 1} + S_{1}^{r - 1} + S_{2}^{r - 1} + \cdots + S_{N}^{r - 1}, \quad r \in \N.$$
A series $a_0 + a_1 + a_2 + \cdots$ is Ces\`{a}ro summable or order $r$, written $(C, r)$ summable, to $L$ if 
$$\lim_{N \to \infty} \frac{r! S_{N}^{r}}{(N + 1)^r} = L.$$
The matrix that corresponds to $(C, r)$ summability is 
$$W  \begin{bmatrix}
g(1) & \0 & \0 & \0 & \cdots\\[3pt]
\0 & g(\tfrac{1}{2}) & \0 & \0 & \cdots\\[3pt]
\0 & \0 & g(\tfrac{1}{3}) & \0 & \cdots\\[3pt]
\0 & \0 & \0 & g(\tfrac{1}{4}) & \cdots\\
\vdots & \vdots & \vdots & \vdots & \ddots
\end{bmatrix} W,$$
where 
$$g(z) = \frac{r! z^r}{(1 + z) (1 + 2 z) \cdots (1 + (r - 1)z)}.$$

\section{The Ces\`{a}ro operator on the space of analytic functions}

Let $\mathcal{O}(\D)$ denote the vector space of all analytic functions on the open unit disk $\D$. It is well known that $\mathcal{O}(\D)$ is a Fr\'{e}chet space when endowed with the topology arising from uniform convergence on compact subsets of $\D$. 
For $f \in \mathcal{O}(\D)$ define 
\begin{equation}\label{integralformCC}
(C f)(z) = \frac{1}{z} \int_{0}^{z} \frac{f(\xi)}{1 - \xi} d \xi, \quad z \in \D,
\end{equation}
where the integral is along any path from $0$ to $z$ in $\D$. Basic theory of integrals shows that $C f \in \mathcal{O}(\D)$. In fact, we have more. 

\begin{Proposition}
The map $f \mapsto C f$ is a linear isomorphism of $\mathcal{O}(\D)$.
\end{Proposition}

\begin{proof}
The linearity of the map $f \mapsto C f$ comes from the linearity of the integral. This map is injective since 
$$\frac{1}{z} \int_{0}^{z} \frac{f(\xi)}{1 - \xi} d \xi \equiv 0 \implies \int_{0}^{z} \frac{f(\xi)}{1 - \xi} d \xi \equiv 0 \implies \frac{f(z)}{1 - z} \equiv 0 \implies f \equiv 0.$$ 
This map is surjective since if $g \in \mathcal{O}(\D)$, then $f(z) = (1 - z) (z g'(z) + g(z))$ belongs to $\mathcal{O}(\D)$ and one can quickly check that $C f = g$. 
\end{proof}

\begin{Corollary}
The inverse of $C: \mathcal{O}(\D) \to \mathcal{O}(\D)$ is given by 
$$(C^{-1} f)(z) = (1 - z) \frac{d}{dz} (z f(z)).$$
\end{Corollary}

Let us connect this seemingly ``new'' linear transformation on $\mathcal{O}(\D)$, defined by an integral, with the Ces\`{a}ro matrix as a linear transformation on $\mathcal{V}$. We do this by looking at the action of $C$ on the functions $(z^n)_{n \geq 0}$. One can check via integration that 
$$C 1 = -\frac{1}{z} \log(1 - z) = 1+\frac{z}{2}+\frac{z^2}{3}+\frac{z^3}{4} + \cdots,$$
$$C z = - \frac{z + \log(1 - z)}{z} = \frac{z}{2}+\frac{z^2}{3}+\frac{z^3}{4}+\frac{z^4}{5} + \cdots,$$
$$C z^2 = - \Big(1 + \frac{z}{2} + \frac{\log(1 - z)}{z}\Big) =  \frac{z^2}{3}+\frac{z^3}{4}+\frac{z^4}{5}+\frac{z^5}{6} + \cdots,$$
and so on. This says that the matrix representation of $C: \mathcal{O}(\D) \to \mathcal{O}(\D)$ with respect to the functions $(z^n)_{n \geq 0}$ is the Ces\`{a}ro matrix 
$$ \begin{bmatrix}
1 & \0 & \0 & \0 & \0 &  \cdots\\[3pt]
\frac{1}{2} & \frac{1}{2} & \0 & \0 & \0 & \cdots\\[3pt]
\frac{1}{3} & \frac{1}{3} & \frac{1}{3} & \0 & \0 & \cdots\\[3pt]
\frac{1}{4} & \frac{1}{4} & \frac{1}{4} & \frac{1}{4} & \0 & \cdots\\[3pt]
\frac{1}{5} & \frac{1}{5} & \frac{1}{5} & \frac{1}{5} & \frac{1}{5} & \cdots\\[-3pt]
\vdots & \vdots & \vdots & \vdots & \vdots & \ddots
\end{bmatrix}.$$

We have already seen the eigenvalues and eigenvectors of the Ces\`{a}ro matrix from Proposition \ref{EVEV}. Let us see them as functions. To do this, we fix a $\lambda \in \C$ and solve the functional equation $C f = \lambda f$. This yields 
$$\frac{1}{z} \int_{0}^{z} \frac{f(\xi)}{1 - \xi} d \xi = \lambda f(z).$$ Multiplying through by $z$ and then differentiating gives us 
$$\frac{f(z)}{1 - z} = \lambda z  f'(z) + \lambda f(z).$$
This differential equation has the solution  
$$f(z) = z^{\frac{1 - \lambda}{\lambda}} (1 - z)^{-\frac{1}{\lambda}}.$$
Since the function above needs to be analytic on $\D$, it must be the case that the exponent on $z$ must be a nonnegative integer and so
$$\frac{1 - \lambda}{\lambda} \in \N_0.$$
Solving for $\lambda$ produces the eigenvalues 
\begin{equation}\label{6554455TT}
\frac{1}{n + 1}, \quad n \in \N_0
\end{equation}
 and the corresponding eigenfunctions 
 \begin{equation}\label{6554455TTT}
f_{n}(z) = \frac{z^n}{(1 - z)^{n + 1}}.
\end{equation}
Notice how this corresponds, via the identification of  the function $z^m \in \mathcal{O}(\D)$ with the basis vector $\vec{e}_m \in \mathcal{V}$ to the eigenvalues and eigenvectors seen  in \eqref{bmmmmm}. Indeed, 
$$f_{n}(z) = \sum_{k = n}^{\infty} {k \choose n} z^n$$
and the Taylor coefficients of $f_n$ are the entries of the vector $\mathbf{b}_n$ from \eqref{bmmmmm}.

We end this section with a re-examination of \eqref{HCe} which relates the Hilbert and Ces\`{a}ro matrices. A change of variable in \eqref{integralformCC} will show that the integral form of $C$ on $\mathcal{O}(\D)$ can be written as 
\begin{equation}\label{ooJJonnBO}
(C f)(z) = \int_{0}^{1} f(t z) \frac{1}{1 - t z} dt.
\end{equation}
Since the matrix entries of the Hilbert matrix $H$ are
$$H_{jk} = \frac{1}{j + k + 1}, \quad j, k \geq 0,$$
we can use this to define a linear transformation $H$ on $\mathcal{O}(\D)$ by 
$$(H f)(z) = \sum_{n = 0}^{\infty} \Big(\sum_{k = 0}^{\infty} \frac{a_k}{n + k + 1}\Big) z^n, \quad f(z) = \sum_{k = 0}^{\infty} a_k z^k.$$
This can be written as integral form as 
$$(H f)(z) = \int_{0}^{1} f(t) \frac{1}{1 - t z} dt.$$
Notice in the integral representation of $C$ from \eqref{ooJJonnBO} that $f(t z)$ appears in the integrand while $f(t)$ (without the $z$) appears in the integrand for $H$. For the functions 
$$e_{n}(z) = z^n, \quad n \geq 0,$$
observe that 
\begin{equation}\label{88hterraced}
(C e_{n})(z) = z^n (H e_{n})(z).
\end{equation}
This says that one can transform the Hilbert matrix to the Ces\`{a}ro matrix as follows: move the $n$th column of $H$ down $n$ places and fill in the $n$ empty slots with zeros. In an analogous way, one can transform $C$ to $H$.

\section{The Ces\`{a}ro operator is bounded on $\ell^2$}\label{cvgsdhfjgksa111}

The Ces\`{a}ro matrix defines a linear transformation $\vec{a} \mapsto C \vec{a}$ on the vector space $\mathcal{V}$ of all complex sequences $\vec{a} = (a_n)_{n \geq 0}$. In fact, it defines an isomorphism of $\mathcal{V}$ onto itself. Now consider the well known Hilbert space
$$\ell^2 :=  \Big\{\vec{a} = (a_n)_{n \geq 0}: \|\vec{a}\| := \Big(\sum_{n = 0}^{\infty} |a_n|^2\Big)^{\frac{1}{2}} < \infty\Big\}$$ with inner product
$$\langle \vec{a}, \vec{b}\rangle = \sum_{n = 0}^{\infty} a_n \overline{b_n}$$ and corresponding norm $\|\vec{a}\| = \sqrt{\langle \vec{a}, \vec{a}\rangle}$.
Is the linear transformation $\vec{a} \mapsto C \vec{a}$ defined/bounded on $\ell^2$? What is 
$$\|C\| = \sup_{\|\vec{a}\| = 1} \|C \vec{a}\|,$$
the norm of this linear transformation on $\ell^2$?

\begin{Theorem}[Brown--Halmos--Shields \cite{MR187085}]\label{iudsfhghdsw}
The linear transformation $\vec{a} \mapsto C \vec{a}$, denoted simply by $C$, defines a bounded operator on $\ell^2$ with $\|C\| = 2$. 
\end{Theorem}

\begin{proof}
Let us outline the original proof from \cite{MR187085}.
The boundedness of $C$ comes from Hardy's inequality \cite{MR1544414} (see also \cite{MR944909}): 
\begin{equation}\label{098787uiujh}
\sum_{n = 1}^{\infty} \Big|\frac{b_0 + b_1 + b_2 + \cdots + b_{n}}{N + 1}\Big|^2 \leq 4 \sum_{n = 1}^{\infty} |b_n|^2.
\end{equation} 
One can also get the boundedness of $C$ from Schur's theorem.
This shows that $C$ has a meaningful (bounded) Hilbert space adjoint 
\begin{equation}\label{098787uiujh}
\begin{bmatrix}
1 & \0 & \0 & \0 & \0 &  \cdots\\[3pt]
\frac{1}{2} & \frac{1}{2} & \0 & \0 & \0 & \cdots\\[3pt]
\frac{1}{3} & \frac{1}{3} & \frac{1}{3} & \0 & \0 & \cdots\\[3pt]
\frac{1}{4} & \frac{1}{4} & \frac{1}{4} & \frac{1}{4} & \0 & \cdots\\[3pt]
\frac{1}{5} & \frac{1}{5} & \frac{1}{5} & \frac{1}{5} & \frac{1}{5} & \cdots\\[-3pt]
\vdots & \vdots & \vdots & \vdots & \vdots & \ddots
\end{bmatrix}^{*} = \begin{bmatrix}
1 & \frac{1}{2} & \frac{1}{3} & \frac{1}{4} & \frac{1}{5} &  \cdots\\[3pt]
\0 & \frac{1}{2} & \frac{1}{3} & \frac{1}{4} & \frac{1}{5} & \cdots\\[3pt]
\0 & \0 & \frac{1}{3} & \frac{1}{4} & \frac{1}{5} & \cdots\\[3pt]
\0 & \0 & \0 & \frac{1}{4} & \frac{1}{5} & \cdots\\[3pt]
\0 & \0 & \0 & \0 & \frac{1}{5} & \cdots\\[-3pt]
\vdots & \vdots & \vdots & \vdots & \vdots & \ddots
\end{bmatrix}.
\end{equation}

The computation of $\|C\|$ goes as follows. A matrix calculation shows that 
$$(I - C)(I - C)^{*} = 
\begin{bmatrix}
0 & \0 & \0 &  \0 & \cdots\\[3pt]
\0 & \frac{1}{2} & \0 & \0 & \cdots\\[3pt]
\0 & \0 & \frac{2}{3} & \0 & \cdots\\[3pt]
\0 & \0 & \0 & \frac{3}{4} & \cdots\\[3pt]
\vdots & \vdots & \vdots & \vdots & \ddots\\
\end{bmatrix}.$$
Now use the $C^{*}$-algebra identity for bounded operators (i.e., $\|A^{*} A\| = \|A\|^2$) on Hilbert space to obtain
\begin{equation}\label{ooKOKooKOJJoo}
 \|I - C\|^2 = \|(I - C) (I - C)^{*}\|  = \sup_{n \geq 0} \frac{n}{n + 1} = 1.
 \end{equation}
Note the use of the fact that the norm of a diagonal operator is the supremum of the absolute values of its entries. 
The triangle inequality now shows that 
$$\|C\| = \|I - (I - C)\| \leq \|I\| + \|I - C\| = 1 + 1 = 2.$$

To prove that $\|C\| = 2$, it is enough to prove that $\|C^{*}\| = 2$. Certainly $\|C^{*}\| = \|C\| \leq 2$. One can argue, as was done in \cite{MR187085}, that for each $a > \tfrac{1}{2}$ we have
$$\|C^{*} (\tfrac{1}{(n + 1)^a})_{n \geq 0}\| \geq \frac{1}{a} \|(\tfrac{1}{(n + 1)^a})_{n \geq 0}\|.$$
Letting $a \to \tfrac{1}{2}$ proves the result. 
\end{proof}

Though the focus of this paper will be on the Ces\`{a}ro operator in the Hilbert space setting of $\ell^2$ (and an associated space of analytic functions -- see below), we would be remiss if we did not make a few remarks about an extension of the previous discussion to the more general setting of 
$$\ell^{p} :=  \Big\{\vec{a} = (a_n)_{n \geq 0}: \|\vec{a}\|_{p} := \Big(\sum_{n = 0}^{\infty} |a_n|^p\Big)^{\frac{1}{p}} < \infty\Big\}, \quad 1 < p < \infty.$$
The boundedness of $C$ on $\ell^p$ for $1 < p < \infty$ comes from another version of Hardy's inequality: 
$$\sum_{n = 0}^{\infty} \Big|\frac{b_0 + b_1 + b_2 + \cdots + b_{n}}{n + 1}\Big|^p \leq q^p \sum_{n = 0}^{\infty} |b_n|^p.$$
In the above, $q = p/(p - 1)$ is the conjugate index to $p$. 
Furthermore, the constant $q^p$ above  is optimal \cite{MR944909} and thus, 
$$\|C\|_{\ell^p \to \ell^p} = q.$$
There are results concerning the pesky cases $p = 1$ and $p = \infty$ \cite{MR806429}.

As a bit of a diversion, let us unpack the fact that 
$$\|I - C\|_{\ell^2 \to \ell^2} = 1$$ used in the proof of Theorem \ref{iudsfhghdsw}. What is $\|I - C\|_{\ell^p \to \ell^p}$? This question was posed in \cite{MR1317938} and answered recently in \cite{MR4366190} with the following:
$$\|I - C\|_{\ell^p \to \ell^p} = 
\begin{cases}
{\displaystyle \frac{1}{p - 1}} & \mbox{if $1 < p \leq 2$},\\[10pt]
{\displaystyle m_{p}^{-\frac{1}{p}}} & \mbox{if $2 < p < \infty$},
\end{cases}
$$
where $m_{p} = \min\{p t^{p - 1} + (1 - t)^p - t^p: 0 \leq t \leq \frac{1}{2}\}$.

\section{The spectral properties of the Ces\`{a}ro operator}\label{spectrusms8}

The spectrum of $C$ on $\ell^2$, denoted by $\sigma(C)$, is the set of $\lambda \in \C$ such that $C - \lambda I$ is not invertible in the bounded operators on $\ell^2$. Recall that $\sigma(C)$ is a nonempty compact subset of $\C$ and that $\sigma(C^{*}) = \{\overline{\lambda}: \lambda \in \sigma(C)\}$. The point spectrum, denoted by $\sigma_{p}(C)$, is the set of eigenvalues of $C$.

\begin{Theorem}[Brown--Halmos--Shields \cite{MR187085}]\label{77uuuh77UUU}
The following hold for the Ces\`{a}ro operator on $\ell^2$.
\begin{enumerate}
\item[(i)]$\sigma_{p}(C) = \varnothing$.
\item[(ii)] $\sigma_{p}(C^{*}) = \{z: |z - 1| < 1\}$
\item[(iii)] $\sigma(C) = \{z: |z - 1| \leq 1\}$.
\end{enumerate}
\end{Theorem}

There is a direct proof of this theorem, using only functional analysis and sequence spaces, in the Brown--Halmos--Shields paper. An alternative proof involves function theory on the disk. To do this, we introduce the {\em Hardy space} 
$$H^2 := \Big\{f(z) = \sum_{n = 0}^{\infty} a_n z^n: (a_n)_{n \geq 0} \in \ell^2\Big\}.$$
Routine estimates show that every such power series $f \in H^2$ has a radius of convergence at least one and thus defines an analytic function on $\D$. The Hardy space is a Hilbert space of analytic functions on $\D$ with the norm and inner product it naturally inherits by identifying a power series $f(z) = \sum_{n = 0}^{\infty} a_n z^n$ in $H^2$ with its Taylor coefficients $(a_n)_{n \geq 0}$ in $\ell^2$. In other words,
$$\|f\| = \Big(\sum_{n = 0}^{\infty} |a_n|^2\Big)^{\tfrac{1}{2}} \quad \mbox{and} \quad \langle f, g\rangle = \sum_{n = 0}^{\infty} a_n \overline{b_n},$$
where $(b_n)_{n \geq 0}$ is the sequence of Taylor coefficients of $g \in H^2$. From \eqref{integralformCC} there is the integral form of the Ces\`{a}ro operator via 
$$(C f)(z) = \frac{1}{z} \int_{0}^{z} \frac{f(\xi)}{1 - \xi} d\xi.$$
This shows that if $f(z) = \sum_{n = 0}^{\infty} a_n z^n$ belongs to $H^2$, then 
$$(C f)(z) = \sum_{n = 0}^{\infty} \Big(\frac{1}{n + 1} \sum_{j = 0}^{n} a_j\Big) z^n.$$ By Hardy's inequality from \eqref{098787uiujh}, the Taylor coefficients of $C f$ belong to $\ell^2$ and thus $C f \in H^2$. We also see that the  linear transformation  $f \mapsto C f$, denoted simply by $C$, defines a bounded operator on $H^2$ with $\|C\| = 2$. Essentially, we are recasting what we know about the Ces\`{a}ro operator on $\ell^2$, written in terms of sequences, as in operator on a Hilbert space of analytic functions on $\D$, written in terms of Taylor coefficients. 

What we gain from changing viewpoints is an enlightening  way of proving Theorem \ref{77uuuh77UUU}. For example,  we can show that $\sigma_{p}(C) = \varnothing$. Indeed, if $f \in H^2$ and $C f = \lambda f$ for some $\lambda \in \C$, then, as discovered in \eqref{6554455TT} and \eqref{6554455TTT}, 
$$\lambda = \frac{1}{n + 1} \quad \mbox{and} \quad f(z) = k \frac{z^n}{(1 - z)^{n + 1}}.$$
For any $n \in \N_0$, one can check that 
$$\frac{1}{(1 - z)^{n + 1}} = 1 + (n + 1) z + \tfrac{1}{2} (n + 2) (n + 1) z^2 + \tfrac{1}{6} (n + 3) (n + 2) (n + 1) z^3 +\cdots $$ and so the sequence of Taylor coefficients for $f$ does not belong to $\ell^2$. Thus, the only way for $f$ to belong to $H^2$ is for $k = 0$. In other words, the Ces\`{a}ro operator on $\ell^2$ has no eigenvalues. 

Continuing with our identification of the Ces\`{a}ro operator as a matrix operator on $\ell^2$ with an integral operator on $H^2$, one can discover from a matrix calculation using the adjoint formula from \eqref{098787uiujh} that 
\begin{equation}\label{bcvvCBCBCcCC}
(C^{*} f)(z) = \sum_{n = 0}^{\infty} \Big(\sum_{j = n}^{\infty} \frac{a_j}{j + 1}\Big) z^n = \frac{1}{1 - z} \int_{z}^{1} f(\xi)d\xi.
\end{equation}
Thus, if we want to compute the eigenvalues of $C^{*}$, we need to find the $\lambda \in \C$ such that $C^{*}f = \lambda f$ for some $f \in H^{2}\setminus \{0\}$. 
This leads us to the integral equation 
$$\frac{1}{1 - z} \int_{z}^{1} f(\xi) d \xi= \lambda f(z)$$ which, in a similar way as before, has solutions 
\begin{equation}\label{KnnNHHcV66}
f(z) = k (1 - z)^{\frac{1 - \lambda}{\lambda}}.
\end{equation}
These functions will belong to $H^2$, i.e., have Taylor coefficients which belong to $\ell^2$, precisely when 
$$\Re\Big( \frac{1 - \lambda}{\lambda}\Big) > -\tfrac{1}{2}.$$
A little algebra will show that 
$$\Re\Big( \frac{1 - \lambda}{\lambda}\Big) > -\tfrac{1}{2} \iff |1 - \lambda| < 1.$$
Thus $\sigma_{p}(C^{*}) = \{z: |z - 1| < 1\}$ as advertised in Theorem \ref{77uuuh77UUU}.

From \eqref{ooKOKooKOJJoo} it follows that $\sigma(I -C) \subset \{z:|z| \leq 1\}$ and thus $\sigma(C) \subset \{z: |z - 1| \leq 1\}$. Hence, since $\{z: |z - 1| < 1\} = \sigma_{p}(C^{*}) \subset \sigma(C^{*}) = \{\overline{w}: w \in \sigma(C)\}$, we obtain the containment 
$$\{z: |z - 1| < 1\} \subset \sigma(C) \subset \{z: |z - 1| \leq 1\}.$$  Equality follows by taking closures. This proves Theorem \ref{77uuuh77UUU}.

The analog of Theorem \ref{77uuuh77UUU} for the Ces\`{a}ro operator on $\ell^p$ is the following. 

\begin{Theorem}
The spectrum of the Ces\`{a}ro operator on $\ell^p$ for $1 < p < \infty$ is the closed disk
$\big\{z: |z - \tfrac{q}{2}| \leq \tfrac{q}{2}\big\}.$
\end{Theorem}

\section{The Ces\`{a}ro operator is hyponormal}\label{HN}

A bounded operator $A$ on a Hilbert space $\mathcal{H}$ is {\em hyponormal} if $A^{*}A - A A^{*} \geq 0$, meaning that 
$\langle (A^{*} A - A A^{*}) \vec{x}, \vec{x}\rangle \geq 0$ for all $\vec{x} \in \mathcal{H}$. The Brown, Halmos, Shields paper \cite{MR187085} contains a proof that the Cesaro operator $C$ on $\ell^2$ is hyponormal. Their proof relies on a classical  result of Sylvester concerning positivity of square matrices and is somewhat technical. Here is a simpler proof by Rhaly which connects to the interesting topic of  {\em posinormal operators}  \cite{MR1291108}. 

\begin{Proposition}
The Ces\`{a}ro operator on $\ell^2$ is hyponormal.
\end{Proposition}

\begin{proof}
Consider the diagonal operator 
$$D = \begin{bmatrix}
\frac{1}{2} & \0 & \0 & \0 & \cdots\\[3pt]
\0 & \frac{2}{3} & \0 & \0 & \cdots\\[3pt]
\0 & \0 & \frac{3}{4} & \0 & \cdots\\[3pt]
\0 & \0 & \0 & \frac{4}{5} & \cdots\\
\vdots & \vdots & \vdots & \vdots & \ddots
\end{bmatrix}.$$
Notice that 
$I - D \geq 0$ (diagonal with positive entries).

A matrix computation will show that 
$C^{*} D C = CC^{*}$ and so for all $\vec{a} \in \ell^2$, 
\begin{align*}
\langle (C^{*} C - C C^{*}) \vec{a}, \vec{a}\rangle & = \langle (C^{*} C - C^{*} D C)\vec{a}, \vec{a}\rangle\\
& = \langle (C - D C) \vec{a}, C \vec{a}\rangle\\
& = \langle (I - D) C \vec{a}, C \vec{a}\rangle\\
& \geq 0.
\end{align*}
This says $C^{*} C - C C^{*} \geq 0$, i.e., $C$ is hyponormal. 
\end{proof}

A bounded operator $A$ on a Hilbert space $\mathcal{H}$ is {\em posinormal} if $A A^{*} = A^{*} P A$ for some positive operator $P$ (called the interrupter). 
We have already seen that $C$ is posinormal. It turns out that $C^{*}$ is posinormal as well. Indeed, 
if $S$ denotes the {\em unilateral shift} on $\ell^2$, i.e., 
$$S := 
 \begin{bmatrix} 
\0 & \0 & \0 & \0 & \cdots\\
1 & \0 & \0  & \0& \cdots\\
\0 & 1 & \0 &  \0& \cdots\\
\0 & \0 & 1 & \0 & \cdots\\
\vdots & \vdots & \vdots & \vdots & \ddots
\end{bmatrix},
$$ one can check that the positive operator 
$$P = (C^{*} - S) (C - S^{*})$$ satisfies 
$C^{*} C =  C P C^{*}$ which makes $C^{*}$ posinormal.

General properties of posinormal operators explored in \cite{MR1291108} applied to the Ces\`{a}ro operator and its adjoint, prove the following. 

\begin{Proposition}
For the Ces\`{a}ro operator on $\ell^2$, we have 
$\operatorname{Ran} C = \operatorname{Ran} C^{*}$ and hence $\operatorname{Ran} C$ is dense in $\ell^2$.
\end{Proposition}

One can also see the density of the range of $C$ on $\ell^2$ more directly from the simple facts, using the matrix representation of $C$, that 
$$\vec{e}_0 = C \vec{e}_0 - C \vec{e}_1,$$
$$\tfrac{1}{2} \vec{e}_1 = C \vec{e}_1 - C \vec{e}_2,$$
$$\tfrac{1}{3} \vec{e}_2 = C \vec{e}_2 - C \vec{e}_3,$$
and so on. Note that though the range of $C$ is dense, it is not closed since otherwise $\operatorname{Ran} C = \ell^2$ and, using the fact that $\ker C = \{0\}$, would mean that $C$ is invertible, which it is not (Theorem \ref{77uuuh77UUU}).

We end our discussion of hyponormality with the following observation. Consider the infinite matrix 
$$A = \begin{bmatrix}
\phantom{-}\frac{1}{3} & \phantom{-}\frac{1}{4} & \phantom{-}\frac{1}{5} & \phantom{-}\frac{1}{6} & \phantom{-}\frac{1}{7} & \phantom{-}\frac{1}{8} \cdots \\[3pt]
 -\frac{2}{3} & \phantom{-}\frac{1}{4} & \phantom{-}\frac{1}{5} &\phantom{-} \frac{1}{6} & \phantom{-}\frac{1}{7} & \phantom{-}\frac{1}{8} \cdots \\[3pt]
 \phantom{-}\0 & -\frac{3}{4} & \phantom{-}\frac{1}{5} & \phantom{-}\frac{1}{6} & \phantom{-}\frac{1}{7} & \phantom{-}\frac{1}{8} \cdots \\[3pt]
 \phantom{-}\0 & \phantom{-}\0 & -\frac{4}{5} & \phantom{-}\frac{1}{6} & \phantom{-}\frac{1}{7} & \phantom{-}\frac{1}{8} \cdots \\[3pt]
 \phantom{-}\0 & \phantom{-}\0 & \phantom{-}\0 & -\frac{5}{6} & \phantom{-}\frac{1}{7} & \phantom{-}\frac{1}{8} \cdots \\[3pt]
\phantom{-} \0 &\phantom{-} \0 & \phantom{-}\0 & \phantom{-}\0 & -\frac{6}{7} & \phantom{-}\frac{1}{8} \cdots \\
 \phantom{-}\vdots & \phantom{-}\vdots & \phantom{-}\vdots & \phantom{-}\vdots & \phantom{-}\vdots & \phantom{-}\ddots 
\end{bmatrix}.$$
An application of Schur's theorem will show that $A$ is a contraction on $\ell^2$ and a matrix computation will reveal that 
$$C^{*} = A C.$$This is a specific case of a more general result which says that for every  hyponormal operator $T$, there is a contraction $W$ such that $T^{*} = W T$. 

\section{The Ces\`{a}ro operator is subnormal}\label{KT}

We know from the previous section that the Ces\`{a}ro operator on $\ell^2$ (or $H^2$) is hyponormal ($C^{*}C - C C^{*} \geq 0$). One of the highlights in the subject is a result of Kriete and Trutt \cite{MR281025} which says that $C$ is subnormal. A bounded operator $A$ on a Hilbert space $\mathcal{H}$ is {\em subnormal} if there is a Hilbert space $\mathcal{K}$ which contains $\mathcal{H}$ (as a closed subspace) and a bounded normal operator $N$ on $\mathcal{K}$ such that $N \mathcal{H} \subset \mathcal{H}$ and $N|_{\mathcal{H}} = A$. The normal operator $N$ is called a {\em normal extension} of $A$ and is never unique. In short, a subnormal operator is the restriction of a normal operator to an invariant subspace. One can show that a subnormal  operator is hyponormal but the converse is not true. 

Some immediate examples of subnormal operators come from the following construction. Start with a positive, finite, compactly supported, Borel measure $\mu$ on $\C$. The multiplication operator $N_{\mu} f = z f$ on $L^2(\mu)$ is bounded and normal since $N_{\mu}^{*} f = \overline{z} f$ and thus $N_{\mu} N_{\mu}^{*} = N_{\mu}^{*} N_{\mu}$. Furthermore, if $\mathcal{H}^{2}(\mu)$ denotes the closure of the analytic polynomials in $L^2(\mu)$, i.e., the closed linear span of $\{1, z, z^2, z^3, \ldots\}$, then the multiplication operator $S_{\mu} f = z f$ is bounded on $\mathcal{H}^{2}(\mu)$ and is subnormal since $\mathcal{H}^{2}(\mu)$ is an invariant subspace of $N_{\mu}$ and $N_{\mu}|_{\mathcal{H}^{2}(\mu)} = S_{\mu}$. See \cite{MR1112128} for a detailed treatment of subnormal operators. 

What makes the subnormality of the Ces\`{a}ro operator $C$ so surprising is that there is no clear normal extension. The Kriete--Trutt result shows that $C$ is unitarily equivalent to $S_{\mu}$ on some $\mathcal{H}^{2}(\mu)$ space for a measure $\mu$ supported in $\overline{\D}$ (the closure of $\D$). This is enough to verify the subnormality of $C$.

 Their construction starts with the functions 
 \begin{equation}\label{oouJUUhhhggGGGGG}
 \phi_{w}(z) = (1 - z)^{\frac{w}{1 - w}}, \quad w \in \D.
 \end{equation}
 Observe that each function $\phi_w$ is analytic on $\D$ and, since the linear fractional transformation
 $$w \mapsto \frac{w}{1 - w}$$ maps $\D$ to $\{z: \Re z > -\tfrac{1}{2}\}$, it follows that  $\phi_w \in H^2$ (the sequence of Taylor coefficients belongs to $\ell^2$). Since 
 $$\phi_{\frac{n}{n + 1}}(z) = (1 - z)^{n}, \quad n \in \N_0,$$ we see that the span of $\{\phi_w: w \in \D\}$ contains the polynomials which constitute a dense set in $H^2$.

  Using the integral form of the adjoint of $C$ from \eqref{bcvvCBCBCcCC} we see that 
  \begin{equation}\label{vvxmnBNXM}
 C^{*} \phi_w = (1 - w) \phi_w.
 \end{equation}
 The reader will recognize these eigenfunctions from \eqref{KnnNHHcV66}.
  For each $f \in H^2$, the function 
 $$(K f)(z) := \langle f, \phi_{\overline{z}}\rangle$$
 is analytic on $\D$ and the set 
 $$\mathcal{K} := \{K f: f \in H^2\}$$ is a vector space. Since $\{\phi_w: w \in \D\}$ has dense linear span in $H^2$, we see that $K f \equiv 0$ if and only if $f \equiv 0$. This allows us to define a Hilbert space norm on $\mathcal{K}$ as 
 $$\|K f\|_{\mathcal{K}} := \|f\|_{H^2}.$$ With this norm, and corresponding inner product, the operator $K$ is unitary from $H^2$ onto $\mathcal{K}$. 
 
 \begin{Proposition}\label{66ysdfgHHGG998110097878}
 For all $f \in H^2$, we have 
 $$(K C f)(z) = (1 - z) (K f)(z), \quad z \in \D.$$ 
 Thus, $C$ is unitarily equivalent to the operator of multiplication by the function $1 - z$ on $\mathcal{K}$. 
 \end{Proposition}
 
 \begin{proof}
 For $f \in H^2$ and $z \in \D$, use \eqref{vvxmnBNXM} to see that 
 \begin{align*}
(K C f)(z) & = 
\langle C f, \phi_{\overline{z}}\rangle_{H^2}\\ & = \langle f, C^{*} \phi_{\overline{z}}\rangle_{H^2} \\
& = \langle f, (1 - \overline{z}) \phi_{\overline{z}}\rangle_{H^2}\\
& = (1 - z) \langle f, \phi_{\overline{z}}\rangle_{H^2}\\
& = (1 - z) (K f)(z). \qedhere
\end{align*}
\end{proof}

Therefore, the Ces\`{a}ro operator is unitarily equivalent to the operator of multiplication by $1 - z$ on a Hilbert space $\mathcal{K}$ of analytic functions on $\D$. A discussion of Kriete and Trutt shows that $\mathcal{K}$ contains the polynomials as a dense set and, more importantly (and quite difficult to prove), there is a positive finite Borel measure $\mu$ on $\overline{\D}$ such that 
\begin{equation}\label{qppoUHggfffTT}
\int_{\D} |p|^2 d \mu = \|p\|^{2}_{\mathcal{K}} \quad \mbox{for all $p \in \C[z]$}.
\end{equation}
 This says that $\mathcal{K} = \mathcal{H}^{2}(\mu)$ and thus, $C$ is unitarily equivalent to $I - S_{\mu}$. This last operator is subnormal and hence $C$ is subnormal. The measure $\mu$ is quite complicated which makes the space $\mathcal{H}^{2}(\mu)$ also very complicated. 

\section{The commutant of the Ces\`{a}ro operator}\label{commuette}

We discussed the infinite matrices $T$ which commute with the Ces\`{a}ro matrix $C$, i.e., $C T = T C$ in Theorem \ref{sdlfuuHHJHJHH}. These are the Hausdorff matrices. However, not all of these Hausdorff matrices define bounded operators on $\ell^2$. In this section we describe those which define {\em bounded} operators on $\ell^2$. 

The commutant of the Ces\`{a}ro operator on $\ell^2$ is defined to be the bounded operators $A$ on $\ell^2$ such that 
$A C = C A$. This set, traditionally denoted by $\{C\}'$, is an algebra that is also closed in the weak operator topology. One can describe $\{C\}'$ via the space $\mathcal{H}^2(\mu)$ used in the previous section to prove that $C$ is subnormal. In \cite{MR281025} Kriete and Trutt showed that $C$ is unitarily equivalent to $I - S_{\mu}$ (multiplication by $1 - z$) on $\mathcal{H}^2(\mu)$ for some measure $\mu$ on $\overline{\D}$. In \cite{MR350489} they went on to show the following. 

\begin{Proposition}\label{77yYYHNFGVfgghghyY}
Let $\mu$ be the measure on $\overline{\D}$ from \eqref{qppoUHggfffTT}.
\begin{enumerate}
\item[(i)] For each $\phi \in H^{\infty}$, the algebra of bounded analytic functions on $\D$, the operator $\phi(S_{\mu})$ is a well defined bounded operator on $\mathcal{H}^2(\mu)$ that is equal to $\phi(S_{\mu}) f = \phi f$, $f \in \mathcal{H}^2(\mu)$. 
\item[(ii)] $\{S_{\mu}\}' = \{\phi(S_{\mu}): \phi \in H^{\infty}\}$.
\end{enumerate}
\end{Proposition}

\begin{proof}
The fact that the operator $f \mapsto \phi f$ is bounded on $\mathcal{H}^2(\mu)$ for all $\phi \in H^{\infty}$ is a detail contained in \cite{MR350489}. Moreover, if $\phi$ is an analytic function on $\D$ which satisfies $\phi \mathcal{H}^2(\mu) \subset \mathcal{H}^{2}(\mu)$, then standard results show that $\phi \in H^{\infty}$. 

Clearly $\phi(S_{\mu})$ is well defined (and is equal to multiplication by $\phi$) for any polynomial $\phi$. If $(\phi_n)_{n \geq 1}$ is a sequence of polynomials that approximate $\phi$ in the weak-$*$ topology on $H^{\infty}$, one can show, as argued in \cite{MR350489}, that $\phi_{n}(S_{\mu}) \to \phi(S_{\mu})$ in the weak operator topology. 

The discussion in the previous paragraph says that 
$\{\phi(S_{\mu}): \phi \in H^{\infty}\} \subset \{S_{\mu}\}'.$
Now let $T \in \{S_{\mu}\}'$. Then for any polynomial $\phi$ we have $\phi(S_{\mu}) T= T \phi(S_{\mu})$. Apply the previous identity to the constant function $1$ to get 
$T \phi = \phi T1.$
Now approximate any $f \in \mathcal{H}^{2}(\mu)$ by a sequence of polynomials in $L^2(\mu)$, and pass to a subsequence if necessary to assume convergence $\mu$-almost everywhere, to see that $T f = \psi f$, where $\psi = T 1$. As discussed earlier in the proof, $\psi \in H^{\infty}$ and $T = \psi(S_{\mu})$. 
\end{proof}

As a corollary to this result we see the following result of Shields and Wallen \cite{MR287352}.

\begin{Theorem}[Shields--Wallen]\label{ShhhhWalll}
For a bounded operator $A$ on $\ell^2$, the following are equivalent. 
\begin{enumerate}
\item[(i)] $A \in \{C\}'$.
\item[(ii)] $A$ belongs to the weak operator closure of $\{p(C): p\in \C[z]\}$.
\item[(ii)] There is a a bounded analytic function $\psi$ on $\{z: |z - 1| < 1\}$ such that $A = \psi(C)$.
\end{enumerate}
\end{Theorem}

Let us reconnect to the Euler matrices $E_{\lambda}$, $\lambda \in \C$, defined in \eqref{Euler}. We know they commute with the Ces\`{a}ro matrix. Which one of these are bounded and thus belong to the commutant of the Ces\`{a}ro operator on $\ell^2$? The answer comes from \cite{MR3342485}. 

\begin{Theorem}
\hfill 
\begin{enumerate}
\item[(i)] If $\lambda \in (\tfrac{1}{2}, 1]$, then $E_{\lambda}$ defines a bounded operator on $\ell^2$ with $\|E_{\lambda}\| = \lambda^{-\frac{1}{2}}$.
\item If $\lambda \in (0, \tfrac{1}{2}]$, then $E_{\lambda}$ defines a bounded operator on $\ell^2$ with $\|E_{\lambda}\| \leq (1 - \lambda^2)^{-\frac{1}{2}}$.
\item[(iii)] If $\lambda \in \C \setminus (0, 1]$, then $E_{\lambda}$ is unbounded on $\ell^2$.
\end{enumerate} 
\end{Theorem}

For $0 < \lambda < 1$,  the matrix $E_{\lambda}$ defines a bounded operator on $\ell^2$ that commutes with the Ces\`{a}ro operator. By Theorem \ref{ShhhhWalll}, $E_{\lambda} = F_{\lambda}(C)$, where $F_{\lambda}$ is a bounded analytic function on $\{z: |z - 1| < 1\}$. We can compute this functions explicitly. Indeed, 
$$F_{\lambda}(C) = W \begin{bmatrix}
F(1) & \0 &\0 & \0 & \cdots\\
\ast & F(\tfrac{1}{2})& \0 & \0 & \cdots\\
\ast & \ast & F(\tfrac{1}{3}) & \0 & \cdots\\
\ast & \ast & \ast & F(\tfrac{1}{4}) & \cdots\\
\vdots & \vdots & \vdots & \vdots & \ddots
\end{bmatrix}W,$$
where $\ast$ denotes an entry that is not important. The function $F_{\lambda}(z) = \lambda^{1/z - 1}$ is a bounded analytic function on $\{z: |z - 1| < 1\}$ with $F_{\lambda}(\tfrac{1}{n}) = \lambda^n$.

In general, from Theorem \ref{sdlfuuHHJHJHH}, the commutant of the Ces\`{a}ro operator on $\ell^2$ is the set of Hausdorff matrices that define bounded operators on $\ell^2$.

\section{The square root of the Ces\`{a}ro operator}\label{sqrt}

The Ces\'{a}ro operator on $\ell^2$ has a (bounded) square root (in fact two of them). The first result along these lines is the following. Below we will use the Taylor series expansion
$$\sqrt{1 - z} = 1-\frac{z}{2}-\frac{z^2}{8}-\frac{z^3}{16}-\frac{5 z^4}{128}-\frac{7
   z^5}{256} \cdots$$
 Notice that the coefficient sequence of $\sqrt{1 - z}$ is absolutely summable.

\begin{Theorem}[Mashreghi--Ptak--Ross \cite{MPR}]
 The following are equivalent for a bounded operator $A$ on $\ell^2$. 
\begin{enumerate}
\item[(i)] $A^2 = C$.
\item[(ii)]
\begin{equation}\label{1110299YU}
A = \pm \Big(I - \tfrac{1}{2} (I - C) -  \tfrac{1}{8} (I - C)^2 - \tfrac{1}{16}  (I - C)^3 + \cdots\Big),
\end{equation}
where the series above converges in operator norm.
\end{enumerate}
\end{Theorem}

\begin{proof}
$(ii) \implies (i)$ follows from multiplication of series. To see that $(i) \implies (ii)$, observe that if $A$ is bounded on $\ell^2$ and $A^2 = C$, then $A C = C A$. By Proposition \ref{66ysdfgHHGG998110097878}, this implies that the operator $K A K^{*}$ commutes with $I - S_{\mu} $ and thus, by Proposition \ref{77yYYHNFGVfgghghyY}, $K A K^{*} = \phi(S_{\mu})$ for some $\phi \in H^{\infty}$. But since $K A K^{*}$ must also be a bounded square root of $\phi(S_{\mu})$, it must be the case that $\phi(z) = \sqrt{1 - z}$ (or $\phi(z) = -\sqrt{1 - z}$). Now use the fact that 
the coefficient sequence of $\sqrt{1 - z}$ is absolutely summable to see that 
   $$\phi(S_{\mu}) = I - \tfrac{1}{2} S_{\mu} - \tfrac{1}{8} S_{\mu}^2 - \tfrac{1}{16} S_{\mu}^{3} - \cdots$$
   converges in operator norm. Since $K^{*} S_{\mu} K = I - C$ we see that 
   $$A = \pm \Big(I - \tfrac{1}{2} (I - C) -  \tfrac{1}{8} (I - C)^2 - \tfrac{1}{16}  (I - C)^3 + \cdots\Big),$$
   which completes the proof. 
\end{proof}

The theorem above may seem somewhat unsatisfactory since it does not really yield a formula (in some way) for the two square roots of $C$. Towards finding a specific formula for the square root of $C$, recall from \eqref{Haus} that  $C = W D W$, where 
$$W = \begin{bmatrix}
1 & \0 & \0 & \0 & \0 &  \cdots\\[3pt]
1 & -1 & \0 & \0 & \0 & \cdots\\[3pt]
1 & -2 & 1 & \0 & \0 & \cdots\\[3pt]
1 & -3 & 3 & -1 & \0 & \cdots\\[3pt]
1 &-4 & 6 & -4 & 1 & \cdots\\[-3pt]
\vdots & \vdots & \vdots & \vdots & \vdots & \ddots
\end{bmatrix},
\quad D = \begin{bmatrix}
1 & \0 &\0 & \0 & \0 & \cdots\\
\0 & \frac{1}{2}& \0 & \0 & \0 &  \cdots\\
\0 & \0 & \frac{1}{3} & \0 & \0 & \cdots\\
\0 & \0 & \0 & \frac{1}{4}& \0 &  \cdots\\
\0 & \0 & \0 & \0 & \frac{1}{5} & \cdots\\
\vdots & \vdots & \vdots & \vdots & \vdots & \ddots
\end{bmatrix},$$
and $W^2 = I$. 
It is important to notice here that $W$ does not define a bounded operator on $\ell^2$ and the above is merely matrix multiplication, which is justified since the matrices involved are lower triangular. 

From here one can immediately see square roots of the Ces\`{a}ro matrix, namely 
\begin{equation}\label{FFggfffGFGGGGGFFG}
A^{\sigma} = W \begin{bmatrix}
\pm 1 & \0 &\0 & \0 & \cdots\\
\0 & \pm\sqrt{\frac{1}{2}}& \0 & \0 & \cdots\\
\0 & \0 & \pm \sqrt{\frac{1}{3}} & \0 & \cdots\\
\0 & \0 & \0 & \pm \sqrt{\frac{1}{4}} & \cdots\\
\vdots & \vdots & \vdots & \vdots & \ddots
\end{bmatrix}W,
\end{equation}
 where the signs along the diagonal of the middle matrix is determined by  $\sigma: \N \to \{\pm 1\}$. In fact, $A^{\sigma}$ can be computed in closed form as 
 $$A^{\sigma}_{ij} =
\begin{cases}
{\displaystyle  {i \choose j} \sum_{\ell = 0}^{i - j} (-1)^{\ell} \sigma(\ell + j + 1) \frac{1}{\sqrt{\ell + j  + 1}} {i - j \choose  \ell}} & i \geq j,\\
 0 & i < j.
 \end{cases}$$

\begin{Theorem}[Hupert--Leggett \cite{MR979593}]
For an infinite lower-triangular matrix $A$, the following are equivalent. 
\begin{enumerate}
\item[(i)] $A^2 = C$.
\item[(ii)] $A = A^{\sigma}$ for some $\sigma: \N \to \{\pm 1\}$.
\end{enumerate}
\end{Theorem}

 The matrices from \eqref{1110299YU} are all lower triangular and one can see that they are  
 $${\tiny  \begin{bmatrix}
1 & \0 & \0 & \0 & \0 &  \cdots\\[3pt]
1 & -1 & \0 & \0 & \0 & \cdots\\[3pt]
1 & -2 & 1 & \0 & \0 & \cdots\\[3pt]
1 & -3 & 3 & -1 & \0 & \cdots\\[3pt]
1 &-4 & 6 & -4 & 1 & \cdots\\[-3pt]
\vdots & \vdots & \vdots & \vdots & \vdots & \ddots
\end{bmatrix}  \begin{bmatrix}
1 & \0 &\0 & \0 & \cdots\\
\0 & \sqrt{\frac{1}{2}}& \0 & \0 & \cdots\\
\0 & \0 & \sqrt{\frac{1}{3}} & \0 & \cdots\\
\0 & \0 & \0 & \sqrt{\frac{1}{4}} & \cdots\\
\vdots & \vdots & \vdots & \vdots & \ddots
\end{bmatrix}\begin{bmatrix}
1 & \0 & \0 & \0 & \0 &  \cdots\\[3pt]
1 & -1 & \0 & \0 & \0 & \cdots\\[3pt]
1 & -2 & 1 & \0 & \0 & \cdots\\[3pt]
1 & -3 & 3 & -1 & \0 & \cdots\\[3pt]
1 &-4 & 6 & -4 & 1 & \cdots\\[-3pt]
\vdots & \vdots & \vdots & \vdots & \vdots & \ddots
\end{bmatrix}} $$ and 
$${\tiny  \begin{bmatrix}
1 & \0 & \0 & \0 & \0 &  \cdots\\[3pt]
1 & -1 & \0 & \0 & \0 & \cdots\\[3pt]
1 & -2 & 1 & \0 & \0 & \cdots\\[3pt]
1 & -3 & 3 & -1 & \0 & \cdots\\[3pt]
1 &-4 & 6 & -4 & 1 & \cdots\\[-3pt]
\vdots & \vdots & \vdots & \vdots & \vdots & \ddots
\end{bmatrix}  \begin{bmatrix}
-1 &\0 &\0 & \0 & \cdots\\
\0 & -\sqrt{\frac{1}{2}}& \0 & \0 & \cdots\\
\0 & \0 & -\sqrt{\frac{1}{3}} & \0 & \cdots\\
\0 & \0 & \0 & -\sqrt{\frac{1}{4}} & \cdots\\
\vdots & \vdots & \vdots & \vdots & \ddots
\end{bmatrix}\begin{bmatrix}
1 & \0 & \0 & \0 & \0 &  \cdots\\[3pt]
1 & -1 & \0 & \0 & \0 & \cdots\\[3pt]
1 & -2 & 1 & \0 & \0 & \cdots\\[3pt]
1 & -3 & 3 & -1 & \0 & \cdots\\[3pt]
1 &-4 & 6 & -4 & 1 & \cdots\\[-3pt]
\vdots & \vdots & \vdots & \vdots & \vdots & \ddots
\end{bmatrix}}.$$
Thus, the two matrices above are the {\em bounded} square roots of the Ces\`{a}ro matrix. The other choices of signs in 
\eqref{FFggfffGFGGGGGFFG} yield square roots of the Ces\`{a}ro matrix but these do not define bounded operators on $\ell^2$. The result above was rediscovered and generalized in \cite{PartGall}.

\section{The Ces\`{a}ro operator and composition operators}\label{CCO}

There is a relationship between the Ces\`{a}ro operator on the Hardy space $H^2$ and composition operators. For any analytic $\phi: \D \to \D$ (often called an {\em analytic self map} of the disk), an application of the Littlewood subordination theorem \cite{Duren} says that the composition operator $C_{\phi} f = f \circ \phi$ defines a bounded operator on $H^2$. 

The first connection between $C_{\phi}$ and $C$ appears in a result of Deddens \cite{MR310691}. If $0 < \alpha < 1$, consider the composition operator $C_{\alpha + (1 - \alpha) z}$. With respect to the standard orthonormal basis $(z^n)_{n \geq 0}$ for $H^2$, the composition operator $C_{\alpha + (1 - \alpha) z}$ has the matrix representation 
$$\begin{bmatrix}
1 & \alpha & \alpha^2 & \alpha^3 & \cdots\\[3pt]
\0 & (1 - \alpha) & 2\alpha (1 - \alpha)& 3 \alpha^2 (1 - \alpha) & \cdots\\[3pt]
\0 & \0 & (1 - \alpha)^2 & 3 \alpha (1 - \alpha)^2 & \cdots \\[3pt]
\0 & \0 & \0 & (1 - \alpha)^3 & \cdots \\
\vdots &\vdots\ & \vdots & \vdots &\ddots
\end{bmatrix}.$$
Furthermore, by matrix multiplication, $C_{\alpha + (1 - \alpha) z}$ commutes with $C^{*}$ and thus
by Theorem \ref{ShhhhWalll}, 
$$C_{\alpha + (1 - \alpha) z} = F(C^{*})$$ for some bounded analytic function on the disk $\{z: |z - 1| < 1\}$. In fact, since 
$$F(C^{*}) = \begin{bmatrix}
F(1) & * & * & * & * &  \cdots\\[3pt]
\0 & F(\frac{1}{2}) &* & * & * & \cdots\\[3pt]
\0 & \0 & F(\frac{1}{3}) & * & * & \cdots\\[3pt]
\0 & \0 & \0 & F(\frac{1}{4}) & * & \cdots\\[3pt]
\0 & \0 & \0 & \0 & F(\frac{1}{5}) & \cdots\\[-3pt]
\vdots & \vdots & \vdots & \vdots & \vdots & \ddots
\end{bmatrix}$$
(the entries with a $\ast$ are unimportant), one can see that 
$F(\tfrac{1}{n}) = (1 - \alpha)^{n - 1}$ and thus, since the sequence $(\tfrac{1}{n})_{n \geq 1}$ does not form a Blaschke sequence in the disk $\{z: |z - 1| < 1\}$, we see that $F(z) = (1 - \alpha)^{1/z - 1}$. Since $C$ and thus $F(C^{*})^{*}$ is subnormal, it follows that $C_{\alpha z + (1 - \alpha)}^{*}$ is subnormal (often called cosubnormal).

Though we will not get into the details here, the Ces\`{a}ro operator also connects to a semigroup of composition operators. Though this dates back to \cite{MR733903, MR310691}, this connection was used most recently in \cite{PartGall} to give some insight into the (complicated) invariant subspace structure for the Ces\`{a}ro operator. Consider the collection of analytic self maps 
$$\phi_{t}(z) = e^{-t} z + (1 - e^{-t}), \quad t \geq 0.$$
One can show that $\phi_{s + t}(z) = \phi_{s}(\phi_{t}(z))$ and that this family of mappings form a {\em holomorphic flow.} Moreover, each of these functions $\phi_t$  maps $\D$ onto another disk which is internally tangent to $\D$ at $z = 1$. The family of corresponding composition operators 
$\{C_{\phi_t}: t \geq 0\}$ forms a strongly continuous semigroup of operators on $H^2$. Using the fact that one can quickly identify the infinitesimal generator of the semigroup, Gallardo and Partington \cite{PartGall} show that 
$$(C^{*} f)(z) = \int_{0}^{\infty} e^{-t} (C_{\phi_t} f)(z) dt.$$
They use this to discuss the invariant subspaces of $C$ in the following Beurling--Lax type of result. 

\begin{Theorem}[Gallardo--Partington \cite{PartGall}]
A closed subspace $\mathcal{M}$ of $H^2$ is invariant for the Ces\`{a}ro operator if and only if $\mathcal{M}^{\perp}$ is invariant for each $C_{\phi_t}$, $t \geq 0$.
\end{Theorem}

\section{Invariant subspaces}\label{ISCOp}

The invariant subspaces of the Ces\`{a}ro operator on $\ell^2$ seem complicated and do not have a complete description. There are various ways to measure the complexities involved. 

As a first order of business, we establish the fact that $C$ is irreducible meaning there are no closed subspaces $\mathcal{M}$ of $H^2$ for which $C \mathcal{M} \subset \mathcal{M}$ and $C^{*} \mathcal{M} \subset \mathcal{M}$. Indeed, if $C$ were reducible, then, via unitary equivalence,  the multiplication operator $M_z$ on the Kriete--Trutt space $H^2(\mu)$ would also be reducible. Some basic facts from operator theory would mean there is an orthogonal projection $P$ on $H^2(\mu)$ for which $P M_{z} = M_{z} P$.  Well known results concerning the commutant of $M_z$ will show that $P = M_{\phi}$ for some $\phi \in H^{\infty}$. However, from the fact that $P^2 = P$, it follows that $M_{\phi^2} = M_{\phi}$ and thus $\phi \equiv 0$ or $\phi \equiv 1$. Thus, $P = 0$ or $P = I$ and so $M_z$, and hence $C$, is irreducible. 

A paper of Kriete and Trutt \cite{MR350489} began to explore some of the complexities of the invariant subspaces of Ces\`{a}ro operator by means of $M_z$ on $\mathcal{H}^2(\mu)$. For example, if  $Z = (z_n)_{n \geq 1} \subset \D$ and 
$$\mathcal{M}(Z) = \{f \in \mathcal{H}^{2}(\mu): f(z_n) = 0 \; \; \mbox{for all $n \geq 1$}\},$$
then $\mathcal{M}(Z)$ is closed in $\mathcal{H}^2(\mu)$ and is an invariant for $M_z$. In the definition of $\mathcal{M}(Z)$ above, we impose the condition that the derivatives of order $k - 1$ vanish at $z_n$ if the zero $z_n$ is repeated $k$ times in the sequence. Via the functions $\phi_w$ from \eqref{oouJUUhhhggGGGGG} and the unitary equivalence between $C$ and $I - M_z$ given in Proposition \ref{66ysdfgHHGG998110097878}, the space $\mathcal{M}(Z)$ gives rise to the $C$-invariant subspace 
$$\Big(\bigvee\{ \phi_{z_n}: n \geq 1\}\Big)^{\perp}.$$
Of course there is the issue as to whether $\mathcal{M}(Z) \not = \{0\}$, in other words, $Z$ is a zero set for $\mathcal{H}^2(\mu)$, but this issue seems very far from being resolved. Indeed, as pointed out in \cite{MR350489}, one can create two sequences $Z_1$ and $Z_2$ such that $\mathcal{M}(Z_1) \not = \{0\}$ and $\mathcal{M}(Z_2) \not = \{0\}$ but $\mathcal{M}(Z_1) \cap \mathcal{M}(Z_2) = \{0\}$. This indicates that the zeros of $H^{2}(\mu)$ can be quite wild and behave more like those of the Bergman space than those of the Hardy and Dirichlet spaces. 

Further indications of the complexities of the invariant subspaces of the Ces\`{a}ro operator come from the paper \cite{PartGall} where they show that  the invariant subspaces of $C$ are in one-to-one correspondence with the subspaces of $L^2(\R, e^{-2(e^x - 1)}dx)$ which are invariant under the right shift semigroup $\{S_t, t \geq 0\}$, where $S_{t} f(x) = f(x - t)$. Moreover, these right shift invariant subspaces are known to be quite complicated and a complete description of them is unknown.

\section{The continuous Ces\`{a}ro operator}

There is a continuous version of the Ces\`{a}ro operator on $L^2[0, 1]$ (inspired by the integral formula in \eqref{integralformCC}) that was explored in \cite{MR187085}. Define the operator 
$$C_{1}: L^2[0, 1] \to L^{2}[0, 1], \quad (C_1 f)(x) = \frac{1}{x} \int_{0}^{x} f(t) dt.$$
Classical real analysis says that $(C_{1} f)(x)$ defines a continuous function on $[0, 1]$. Furthermore,  an application of Schur's integral test, as was done in \cite{MR187085}, shows that $f \mapsto C_1 f$, denoted by $C_1$, is a bounded operator on $L^{2}[0, 1]$. One can also use an integral version of Hardy's inequality, namely
$$\int_{0}^{1} |(C_1 f)(x)|^2 dx \leq 4 \int_{0}^{1} |f(t)|^2 dt \quad \mbox{for all $f \in L^{2}[0, 1]$,}$$
to show that $\|C_1\| \leq 2$. In fact, $\|C_1\| = 2$.

Unlike its discrete counterpart on $\ell^2$, the continuous Ces\`{a}ro operator $C_1$ has eigenvalues. Indeed, 
$$(C_{1} f)(x) = \lambda f(x) \iff \lambda \frac{d}{dx}(x f(x)) = f(x).$$
The solutions to this differential (Euler) equation are scalar multiplies of the functions 
\begin{equation}\label{oijhbsdfgi7eEV}
f_{\lambda}(x) = x^{\frac{1}{\lambda} - 1} \quad \mbox{where $\Re(\frac{1}{\lambda}) > \frac{1}{2}$}.
\end{equation} From here one can show that the eigenvalues of $C_1$ are the open disk $\{z: |z - 1| < 1\}$ and, in a similar way that was done for the discrete Ces\`{a}ro operator, $\sigma(C_{1}) = \{z: |z - 1| \leq 1\}$.

There is also an $L^p[0, 1]$ version of all of this where one can show that $C_{1}$ is a bounded operator on $L^{p}[0, 1]$ when $1 < p < \infty$ with norm equal to $p/(p - 1)$. Moreover, the  operator $C_1$ on $L^p[0, 1]$ has eigenvalues $\{z: |z - q/2| < q/2\}$ and spectrum the closure of this disk. There is also the resolvent formula \cite{MR322569}
$$(\lambda I - C_{1})^{-1} = \frac{1}{\lambda} I + \frac{1}{\lambda^2} P_{\lambda},  \quad |\lambda - 1| > 1,$$
where 
$$(P_{\lambda} f)(x) = \int_{0}^{1} f(x t) t^{-1/\lambda} dt.$$

One can also show that 
$$(C_{1}^{*} f)(x) = \int_{x}^{1} \frac{f(t)}{t} dt$$
and there is an important connection between $C_{1}^{*}$ and the unilateral shift that was discovered in \cite{MR187085}. By means of Laplace transforms and Hardy spaces of the right-half plane, they prove the following.

\begin{Theorem}[Brown--Halmos--Shields]\label{5558887979}
The operator $I - C_{1}^{*}$ is unitarily equivalent to the shift operator $(S f)(z) = z f(z)$ on the Hardy space $H^2$. 
\end{Theorem}

As an interesting corollary to this, one can prove the following, which is quite cumbersome to do directly. 

\begin{Corollary}
The operator $I - C_{1}^{*}$ is an isometry on $L^2[0, 1]$.
\end{Corollary}

From Theorem \ref{5558887979} one could obtain the invariant subspaces of $C_{1}$, or $C_{1}^{*}$, via Beurling's characterization of the invariant subspaces of the shift on $H^2$ \cite{Duren}. However, the correspondence in \cite{MR187085} between the invariant subspaces of the shift on $H^2$ the the invariant subspaces of $I - C_{1}^{*}$ is hidden by unitary operators making the invariant subspaces somewhat difficult to understand. Nevertheless, Nikolskii \cite[p.~37]{N1} was able to use Hardy space theory to describe the invariant subspaces as Theorem \ref{iinIkskAgler} below. A direct approach was done recently by Agler and McCarthy in \cite{AglMc}. Their description of the invariant subspaces of $C_{1}$ goes as follows. 

Let $W = \{z: \Re z > -\tfrac{1}{2}\}$ and for a {\em finite} subset $B$ of $W$, let 
$$\mathcal{X}(B) := \operatorname{span} \{x^{b}: b \in B\}.$$
Since 
$$C_{1} x^b = \frac{1}{b + 1} x^b$$ we see that $\mathcal{X}(B)$ is an invariant subspace of $C_{1}$. The condition that $\Re b > -\tfrac{1}{2}$ assures that $x^b \in L^{2}[0, 1]$. Agler and McCarthy call such $\mathcal{X}(B)$ {\em finite monomial spaces}. For a sequence $\mathcal{B} := (B_n)_{n \geq 1}$ of finite subsets of $W$ define 
$$\mathcal{X}_{\mathcal{B}} := \Big\{f \in L^{2}[0, 1]: \lim_{n \to \infty} \operatorname{dist}(f, \mathcal{X}(B_n))= 0\Big\}.$$ One can verify that $\mathcal{X}_{\mathcal{B}}$ is a closed subspace of $L^{2}[0, 1]$ that is invariant for $C_{1}$. These are called {\em monomial spaces}. 

\begin{Theorem}\label{iinIkskAgler}
Let $\mathcal{X}$ be a closed nonzero subspace of $L^{2}[0, 1]$. Then $\mathcal{X}$ is invariant for $C_{1}$ if and only if $\mathcal{X}$ is a monomial space. 
\end{Theorem}

There is a version of the continuous Ces\`{a}ro operator that is defined on $L^{2}(0, \infty)$ by 
$$(C_{\infty} f)(x) = \frac{1}{x} \int_{0}^{x} f(t) dt.$$
An analysis with a  version of Hardy's inequality for $L^{2}(0, \infty)$ shows that $\|C_{\infty}\| = 2$. 
Formally, the eigenvectors of $C_{\infty}$ can be computed as done with $C_{1}$ (see \eqref{oijhbsdfgi7eEV}) but none of them belong to $L^{2}(0, \infty)$ and thus $\sigma_{p}(C_{\infty}) = \varnothing$. There is also the resolvent formula 
$$(\lambda I - C_{\infty})^{-1} = \frac{1}{\lambda} I + \frac{1}{\lambda^2} P_{\lambda},  \quad |\lambda - 1| > 1.$$
One also shows that $\sigma(C_{\infty})$ is the circle $\{z: |z - 1| = 1\}$. Somewhat similar to the situation with $C_{1}$, there is the following.

\begin{Theorem}[Brown--Halmos--Shields \cite{MR187085}]
The operator $I - C_{\infty}^{*}$ is unitarily equivalent to the bilateral shift operator $(M f)(\xi) = \xi f(\xi)$ on $L^{2}(\T)$. 
\end{Theorem}

\section{The Ces\`{a}ro operator in other settings}\label{Bergman}

There are a host of results which document when the Ces\`{a}ro operator 
$$(C f)(z) = \frac{1}{z} \int_{0}^{z} \frac{f(\xi)}{1 - \xi} d \xi$$ and its adjoint
$$(A f)(z) = \frac{1}{1 - z}\int_{z}^{1} f(\xi) d \xi$$ define bounded linear transformations on various Banach and Fr\'{e}chet spaces of analytic functions on $\D$. In the Hilbert space $H^2$ we have $A = C^{*}$. 
For a Banach space on analytic functions on $\D$, the adjoint operator acts on the dual space and not the original space. In some cases, one can even compute the spectrum of the operators $C$ and $A$. We document a selection of these results and then provide some guidance where to discover more about this. 

The first set of results deal with the Hardy spaces $H^p$ with $0 < p < \infty$ of analytic function $f$ on $\D$ for which 
$$\|f\|_{H^p} := \Big(\sup_{0 < r < 1} \int_{\T} |f(r \xi)|^p dm(\xi)\Big)^{\frac{1}{p}} < \infty.$$
Note that when $p = 2$, we recover the Hardy space $H^2$ discussed in previous sections. 
It is well known \cite{Duren} that when $1 \leq p < \infty$, the Hardy space $H^p$ is a Banach space with norm $\| \cdot\|_{H^p}$. When $0 < p < 1$, $H^p$ is a Fr\'{e}chet space with the metric $\|\cdot\|_{H^p}^{p}$. As a consequence of \cite{MR11137} and \cite{MR14154}, one can prove that when $1 \leq p < \infty$, the Ces\`{a}ro operator $C$ defines a bounded operator on $H^p$. In \cite{MR897683}, Siskakis used semigroups of weighted composition operators (see below) to prove the following. 

\begin{Theorem}[Siskakis]\label{uygiIIUYTFGHBVGHYTRTYU}
Consider the Ces\`{a}ro operator $C$ on $H^p$. 
\begin{enumerate}
\item[(i)] If $2 \leq p < \infty$, then $\|C\|_{H^p \to H^p} = p$ and the spectrum of $C$ is the closed disk $\{z: |z - \tfrac{1}{p}| \leq \tfrac{1}{p}\}$.
\item[(ii)] If $1 \leq p < 2$, then $p \leq \|C\|_{H^p \to H^p} \leq 2$ and the spectrum of $C$ contains the closed disk  $\{z: |z - \tfrac{1}{p}| \leq \tfrac{1}{p}\}$.
\item[(iii)] If $1 < p < \infty$, then $\|A\|_{H^p \to H^p} = p/(p - 1)$.
\end{enumerate}
\end{Theorem}

The main driver of Theorem \ref{uygiIIUYTFGHBVGHYTRTYU}, and for many other results which were to follow, is the following observation which is interesting in its own right. The set of maps
$$\phi_{t}(z) = \frac{e^{-t} z}{(e^{-t} - 1) z + 1}, \quad t \geq 0,$$
form a semi-group of analytic self maps of $\D$. If we choose the path of integration $t \mapsto \phi_{t}(z)$ in the formula 
$$(C f)(z) = \frac{1}{z} \int_{0}^{z} \frac{f(\xi)}{1 - \xi} d \xi,$$ we obtain
\begin{align*}
(C f)(z) & =  \frac{1}{z} \int_{0}^{z} \frac{f(\xi)}{1 - \xi} d \xi\\
& = -\frac{1}{z} \int_{0}^{\infty} \frac{(f \circ \phi_t)(z)}{1 - \phi_t(z)} \frac{\partial \phi_t(z)}{\partial t} dt\\
& = \int_{0}^{\infty} \frac{\phi_{t}(z)}{z} f \circ \phi_t(z) dt.
\end{align*}
Siskakis proved that when $2 \leq p < \infty$, the linear transformations
$$(S_{t} f)(z) = \frac{\phi_{t}(z)}{z} f \circ \phi_t(z)$$ define bounded operators on $H^p$ with 
$\|S_{t}\|_{H^p \to H^p} \leq e^{-t/p}$. This estimate, along with the integration formula above, play an important role in estimating $\|C\|_{H^p \to H^p}$. More delicate estimates are needed for the $1 \leq p < 2$ case. 

One can extend, at least the boundedness result from  Theorem \ref{uygiIIUYTFGHBVGHYTRTYU}, as follows \cite{MR1104399}. 

\begin{Theorem}[Miao]
For $0 < p < 1$, the Ces\`{a}ro operator on $H^p$ is bounded. 
\end{Theorem}

Let us make it clear that in this case  we are not in a Banach space (but in a Fr\'{e}chet space), the phrase ``$C$ is bounded on $H^p$" means there is a constant $k_p > 0$ such that 
$$\int_{\T} |C f|^p dm \leq k_{p} \int_{\T} |f|^p dm \quad \mbox{for all $f \in H^p$}.$$

Using a similar semigroup analysis, Siskakis discussed the Ces\`{a}ro operator on the Bergman spaces $A^p$, $1 \leq p \leq \infty$, of analytic function $f$ on $\D$ which satisfy 
$$\|f\|_{A^p} = \Big(\int_{\D} |f|^p dA\Big)^{\frac{1}{p}} < \infty.$$
In the above, $dA$ denotes planar area measure. The main result from \cite{MR1407335} is the following. 

\begin{Theorem}[Siskakis]
Consider the Ces\`{a}ro operator $C$ on $A^p$. 
\begin{enumerate}
\item[(i)] If $4 \leq p < \infty$. then $\|C\|_{A^p \to A^p} = \tfrac{p}{2}$ and the spectrum of $C$ is the closed  disk $\{z: |z - \tfrac{p}{4}| \leq \tfrac{p}{4}\}$.
\item[(ii)] If $1 \leq p < 4$, then $\tfrac{p}{2} \leq \|C\|_{A^p \to A^p} \leq 2$ and the spectrum of $C$ contains the closed disk $\{z: |z - \tfrac{p}{4}| \leq \tfrac{p}{4}\}$.
\end{enumerate}
\end{Theorem}

As one can imagine, the Ces\`{a}ro operator has been studied for many other spaces of analytic functions on $\D$ and it is impossible to survey them all. We direct the reader to the papers  \cite{MR1903676, MR4365517, MR2390921} which contain useful references to the large literature on this subject. 

\section{Generalized Ces\`{a}ro operators}\label{g}

Notice how the Ces\`{a}ro operator (in integral form) can be written as 
$$(C f)(z) = \frac{1}{z} \int_{0}^{z} f(\xi) g'(\xi) d \xi, \quad \mbox{where} \quad g(z) = \log\Big(\frac{1}{1 - z}\Big).$$ For a general $g \in \mathcal{O}(\D)$, one can define the linear transformation on $\mathcal{O}(\D)$ by 
$$(C_g f)(z) =  \frac{1}{z} \int_{0}^{z} f(\xi) g'(\xi) d \xi.$$ 
This indicates how one can generalize the classical Ces\`{a}ro operator on $H^2$ or other spaces of analytic functions. 
It was shown in \cite{MR454017} (see also \cite{MR1869606}) that $C_g$ defines a bounded operator on $H^2$ if and only if $g$ is of bounded mean oscillation.   Notice how the logarithm function $g$ above is of bounded mean oscillation.

The initial papers mentioned above determine when a generalized Ces\`{a}ro operator is bounded on $H^2$. The study of $C_{g}$ continues with the papers \cite{MR4349219, MR2390921} which explore the spectral properties of $C_g$. Other operator theory results such as invariant subspaces (or at least results showing the complexity of them) still need to be developed and this field seems to be wide open and worthy of further exploration. 

We end this paper with a segue to another paper on  Rhaly matrices that will appear in this volume. This is inspired by the relationship between the Hilbert and Ces\`{a}ro matrices explored in \eqref{88hterraced}. A simple integral substitution shows that the  generalized Ces\`{a}ro operator $C_g$ as a linear transformation on $\mathcal{O}(\D)$ can be written in integral form as 
$$(C_{g}f) (z) = \int_{0}^{1} f(t z) g'(t z) dt.$$
With the functions $e_{n}(z)$ defined by 
$$e_{n}(z) = z^n, \quad n \geq 0,$$ observe that if 
$$g(z) = \sum_{k = 0}^{\infty} g_k z^k,$$ then 
\begin{equation}\label{I1}
(C_g e_{n})(z) = z^{n} \sum_{k = 1}^{\infty} \frac{k}{k + n} g_k z^{k - 1}.
\end{equation}
Thus, the matrix representation of $C_g$ with respect to $(e_n)_{n \geq 0}$ is 
$$[C_{g}] = \begin{bmatrix}
g_1 \tfrac{1}{0 + 1} & 0 & 0 & 0 & 0 & \cdots\\
g_{2} \tfrac{2}{0 + 2} & g_1 \tfrac{1}{1 + 1} & 0 & 0 & 0 & \cdots\\
g_{3} \tfrac{3}{0 + 3} & g_{2} \tfrac{2}{1 + 2} & g_{1} \tfrac{1}{2 + 1} & 0 & 0 & \cdots\\
g_{4} \tfrac{4}{0 + 4} & g_{3} \tfrac{3}{1 + 3} & g_{2} \tfrac{2}{2 + 2} & g_{1} \tfrac{1}{3 + 1} & 0 & \cdots\\
g_{5} \tfrac{5}{0 + 5} & g_{4} \tfrac{4}{1 + 4} & g_{3} \tfrac{3}{2 + 3} & g_{2} \tfrac{2}{3 + 2} & g_{1} \tfrac{1}{1 + 4} & \cdots\\
\vdots & \vdots & \vdots & \vdots & \vdots & \ddots
\end{bmatrix}.$$
When 
$$g(z) = \log \frac{1}{1 - z} = z + \tfrac{1}{2} z^2 + \tfrac{1}{3} z^3 + \cdots,$$
$[C_g]$ becomes the classical Ces\`{a}ro matrix.

For $g \in \mathcal{O}(\D)$, there is the generalized Hilbert operator $H_{g}$ on $\mathcal{O}(\D)$ defined by 
$$(H_{g} f)(z) = \int_{0}^{1} f(t) g'(t z) dt.$$
Notice the $f(t z)$ in the integrand of $C_g$ and $f(t)$ (without the $z$) in the integrand of $H_g$. A calculation shows that 
\begin{equation}\label{I2}
(H_{g} e_n)(z) = \sum_{k = 1}^{\infty} \frac{k}{n + k} g_k z^{k - 1}
\end{equation}
and thus, with respect to $(e_n)_{n \geq 0}$, the matrix representation of $H_g$ is 
$$[H_g] = \begin{bmatrix}
g_1 \tfrac{1}{0 + 1} & g_1 \tfrac{1}{1 + 1} & g_{1} \tfrac{1}{2 + 1} &g_1 \tfrac{1}{3 + 1} & g_1 \tfrac{1}{4 + 1} & \cdots\\
g_{2} \tfrac{2}{0 + 2} & g_{2} \tfrac{2}{1 + 2} & g_{2} \tfrac{2}{2 + 2} & g_{2} \tfrac{2}{3 + 2} & g_2 \tfrac{2}{4 + 2} & \cdots\\
g_{3} \tfrac{3}{0 + 3} & _{3} \tfrac{3}{1 + 3} & g_{3} \tfrac{3}{2 + 3} & g_{3} \tfrac{3}{3 + 3} & g_{3} \tfrac{3}{4 + 3} & \cdots\\
g_{4} \tfrac{4}{0 + 4}  &  g_{4} \tfrac{4}{1 + 4} &  g_{4} \tfrac{4}{2 + 4} &  g_{4} \tfrac{4}{3 + 4} &  g_{4} \tfrac{4}{4 + 4} & \cdots\\
g_{5} \tfrac{5}{0 + 5} & g_{5} \tfrac{5}{1 + 5} & g_{5} \tfrac{5}{2 + 5} & g_{5} \tfrac{5}{3 + 5} & g_{5} \tfrac{5}{4 + 5} & \cdots\\
\vdots & \vdots & \vdots & \vdots & \vdots & \ddots
\end{bmatrix}$$
As before with $C_g$, when 
$$g(z) = \log \frac{1}{1 - z} = z + \tfrac{1}{2} z^2 + \tfrac{1}{3} z^3 + \cdots,$$
$[H_g]$ becomes the classical Hilbert matrix.

The identities \eqref{I1} and \eqref{I2} show, as before with the Hilbert and Ces\`{a}ro matrices in  \eqref{88hterraced}, that one can transform $[H_{g}]$ into $[C_{g}]$ by moving the $n$th column of $H_g$ down $n$ places and filling in the blank spaces with zeros.

\bibliographystyle{plain}

\bibliography{references}

\end{document}